\documentclass[12pt]{extarticle}
\usepackage{amsmath, amsthm, amssymb}
\usepackage[ngerman,french,english]{babel}
\usepackage[utf8]{inputenc}   
\usepackage[T1]{fontenc}
\usepackage{tikz-cd}
\usetikzlibrary{babel}

\usepackage{bm}

\usepackage{verbatim}

\usepackage{theoremref}
\usepackage[hidelinks]{hyperref}
\usepackage{cleveref}

\usepackage{enumitem}
\usepackage{titlesec}
\titleformat*{\section}{\large\bfseries\filcenter}
\titleformat*{\subsection}{\normalsize\bfseries}

\usepackage{abstract}
\usepackage[margin=3.5cm]{geometry}

\newcommand{\cd}{\mathcal{D}}
\newcommand{\ce}{\mathcal{E}}
\newcommand{\cf}{\mathcal{F}}

\newcommand{\ci}{\mathcal{J}}

\newcommand{\cl}{\mathcal{L}}
\newcommand{\cm}{\mathcal{M}}

\newcommand{\co}{\mathcal{O}}

\newcommand{\cq}{\mathcal{Q}}
\newcommand{\cs}{\mathcal{S}}
\newcommand{\cz}{\mathcal{Z}}

\newcommand{\ba}{\mathbf{A}} 
 
\newcommand{\bp}{\mathbf{P}} 
\newcommand{\bz}{\mathbf{Z}} 

\newcommand{\A}{\mathrm{A}}
 
\newcommand{\C}{\mathrm{C}}  
 
\newcommand{\E}{\mathrm{E}} 
\newcommand{\F}{\mathrm{F}} 
 
\let\ho\H
\renewcommand{\H}{\mathrm{H}}

\renewcommand{\L}{\mathrm{L}}
\newcommand{\M}{\mathrm{M}}

\newcommand{\R}{\mathrm{R}} 
\let\SS\S
\renewcommand{\S}{\mathrm{S}}
\newcommand{\T}{\mathrm{T}} 
\newcommand{\U}{\mathrm{U}} 
 
\newcommand{\V}{\mathrm{V}} 
\newcommand{\X}{\mathrm{X}} 
 
\newcommand{\Z}{\mathrm{Z}} 

\newcommand{\isom}{ \ \rotatebox[origin=c]{90}{$\thicksim$}}

\newcommand{\Aut}{\operatorname{Aut}}

\newcommand{\Ext}{\operatorname{Ext}}
\newcommand{\SExt}{\mathcal{E}xt}
\newcommand{\Hom}{\operatorname{Hom}}

\newcommand{\GL}{\operatorname{GL}}

\newcommand{\SEnd}{\mathcal{E}nd}
\newcommand{\SHom}{\mathcal{H}om}

\newcommand{\Quot}{\operatorname{Quot}}

\newcommand{\vir}{\operatorname{vir}}

\newcommand{\git}{\mathbin{/\mkern-6mu/}}

\theoremstyle{plain}
\newtheorem{theorem}{Theorem}[section]
\newtheorem{theorem*}{Theorem}[]
\newtheorem{proposition}{Proposition}[section]
\newtheorem{lemma}{Lemma}[section]
\newtheorem{corollary}{Corollary}[section]
\newtheorem{conjecture}{Conjecture}[section]
\newtheorem{conjecture*}{Conjecture}

\theoremstyle{definition}
\newtheorem{remark}{Remark}[section]
\newtheorem{example}{Example}[section]

\begin{document}

\title{\textbf{On the Quot scheme \texorpdfstring{$\bm{\mathrm{Quot}^{l}_{\mathrm{S}}(\ce)}$}{Quot}}}
\author{SAMUEL STARK}
\date{}
\maketitle

We study the geometry of the Quot scheme $\Quot^{l}_{\S}(\ce)$ of length $l$ coherent sheaf 
quotients of a locally free sheaf $\ce$ on a smooth projective surface $\S$. 
In particular, we investigate the nature of its singularities, its intersection theory, 
and the cohomology of sheaves on $\Quot^l_{\S}(\ce)$.

\section{Introduction}

Let $\ce$ be a locally free sheaf of rank $r$ on a smooth projective surface $\S$, 
and consider the Quot scheme $\Quot^{l}_{\S}(\ce)$ of length $l$ coherent sheaf quotients of $\ce$. 
Defined by Grothendieck in his seminal paper \cite{Gro}, 
these schemes are fundamental objects of study in algebraic geometry. 
Grothendieck related the schemes $\Quot^{l}_{\S}(\ce)$ to the Chow scheme $\S^{(l)}$ of effective $0$-cycles of degree $l$ on $\S$. 
The Hilbert scheme of points $\S^{[l]}=\Quot^{l}_{\S}(\co)$ is smooth, and in fact a resolution of singularities of $\S^{(l)}$; 
the geometry of $\S^{[l]}$ is very rich and has been studied intensely \cite{Bea, ElS, EGL, Fog, Goe, Nak, MOP1}. 
In contrast, the schemes $\Quot^{l}_{\S}(\ce)$ are singular for $l, r\geqslant 2$ and very little is known about them, 
despite the fact that they occur naturally in the study of moduli spaces of sheaves on $\S$ \cite{Bar, GiL, Yos}. 
In this paper, we study some basic aspects of the geometry of these schemes. 

In \SS 2, we set forth our notations, recall some results of Grothendieck \cite{Gro}, and in particular basic properties of the schemes $\Quot^{l}_{\S}(\ce)$, such as
\begin{enumerate}[label=(\roman*)]
\item a quotient sheaf $\ce^{''}$ of $\ce$ induces a closed immersion \label{one} $\iota:\Quot^{l}_{\S}(\ce^{''})\rightarrow\Quot^{l}_{\S}(\ce)$;
\item an invertible sheaf $\cl$ on $\S$ induces an isomorphism \label{two} $\Quot^{l}_{\S}(\ce)\xrightarrow{\sim}\Quot^{l}_{\S}(\ce\otimes\cl)$;
\item the automorphism group of $\ce$ acts naturally on $\Quot^{l}_{\S}(\ce)$. \label{three}
\end{enumerate}
We then introduce the tautological sheaves $\cf^{[l]}$ on $\Quot^{l}_{\S}(\ce)$, 
and the section $s^{[l]}$ of $\cf^{[l]}$ induced by a section $s$ of $\ce\otimes\cf$. 
We discuss some properties of these objects, and especially their compatibility with \ref{one} -- \ref{three}. 
We include a brief description of the obstruction theory of $\Quot^l_{\S}(\ce)$, and the associated virtual class $[\Quot^l_{\S}(\ce)]^{\vir}$.

In \SS 3, we study the local properties of $\Quot^{l}_{\S}(\ce)$. 
\Cref{tangent} provides a formula for the embedding dimension of $\Quot^{l}_{\S}(\ce)$ at a given point. 
Based on Nakajima's ADHM  description of $\Quot^l_{\ba^{2}}(\co^{\oplus r})$, 
we state the following conjectural description of the singularities, which guides our inquiry.
\begin{conjecture*}
The scheme $\Quot^{l}_{\S}(\ce)$ has rational singularities.
\end{conjecture*}
We show that this conjecture holds in the first non-trivial case $l=2$. A weaker (Cohen-Macaulay) form of the above conjecture is implied by a long-standing conjecture on commuting schemes. In the $l=2$ case, we exhibit a modular resolution of singularities of the singular scheme $\Quot^{2}_{\S}(\ce)$, relating $\Quot^{2}_{\S}(\ce)$ to the Hilbert scheme $\bp(\ce)^{[2]}$ (\Cref{lequals2}).
\begin{theorem*}
	(i) There exists a morphism of schemes
	\begin{equation*}
		\rho^{2}:\bp(\ce)^{[2]}\rightarrow\Quot^2_{\S}(\ce)
	\end{equation*}
	taking a length two subscheme $\Z$ of $\bp(\ce)$ to $\ce=p_{\ast} \co(1)\rightarrow p_{\ast} \co_{\Z}(1)$. \\
	(ii) The morphism $\rho^{2}$ is a resolution of singularities of $\Quot^2_{\S}(\ce)$. \\
	(iii) If $\cf$ is a locally free sheaf on $\S$, then
	\begin{equation*}
		\rho^{2\ast}\cf^{[2]}=\cf(1)^{[2]}.
	\end{equation*}
\end{theorem*}

Here $\co(1)$ denotes the canonical invertible sheaf associated to $p:\bp(\ce)\rightarrow\S$.

In \SS 4, we study the intersection theory of $\Quot^{l}_{\S}(\ce)$. 
In the intersection theory of Hilbert schemes of points, 
the Chern classes of tautological sheaves play an important role \cite{EGL, Leh}. 
In the higher rank case, we prove that the morphism $\iota$ of \ref{one} yields relations in intersection theory (\Cref{relations}).
\begin{theorem*}
Let $\ce\rightarrow\ce^{''}$ be a locally free quotient with kernel $\ce^{'}$, 
and consider the section $s\in\H^{0}(\S, \ce\otimes\ce^{'\vee})$ given by the inclusion $\ce'\rightarrow\ce$. 
The induced section $s^{[l]}\in\H^0(\Quot^l_{\S}(\ce), \ce^{'\vee[l]})$ is regular, 
and its zero scheme is $\Quot^l_{\S}(\ce^{''})\subset\Quot^l_{\S}(\ce)$. 
In particular,
\begin{equation*}
\iota_\ast[\Quot^l_{\S}(\ce^{''})] = e(\ce^{'\vee[l]}) \cap [\Quot^l_{\S}(\ce)]
\end{equation*}
holds in $\A_{\ast}(\Quot^l_{\S}(\ce))$. Moreover, we have the vanishing
\begin{equation*}
e(\ce^{\vee[l]}) = 0 \quad in \quad \A_{\ast}(\Quot^l_{\S}(\ce)).
\end{equation*}
\end{theorem*}

This result can be thought of as a higher $l$ generalisation of well-known facts about $\bp(\ce)=\Quot^{1}_{\S}(\ce)$; 
in particular, $e(\ce^{\vee[1]})=0$ is precisely Grothendieck's relation
\begin{equation*}
\sum_{k=0}^{r}(-1)^{r-k}c_{1}(\co(1))^{k}c_{r-k}(\ce)=0.
\end{equation*}
We also prove the corresponding result for the virtual class (\Cref{virclass}).
We then apply \Cref{relations} to the study of tautological integrals, 
i.e. integrals of Chern classes of tautological sheaves. 
In the $r=1$ case, such integrals have attracted a great deal of attention \cite{Bea, EGL, ElS, MOP1, MOP2, Voi}. 
In a recent paper \cite{OP}, Oprea and Pandharipande study virtual tautological integrals over $\Quot^{l}_{\S}(\co^{\oplus r})$; 
using \Cref{virclass}, the results of \cite{OP} were extended to $\Quot^{l}_{\S}(\ce)$ for arbitrary $\ce$ \cite{Sta}. 
(The virtual integrals are much easier to study, but contain less information, and their geometric meaning is obscure.)

In the study of tautological integrals over $\S^{[l]}$, the Segre integrals have been computed in the papers \cite{MOP1, MOP2}. 
In the first non-trivial case $l=2$, the Segre integral is given by Severi's double point formula \cite{Sev}. 
We extend the latter formula to higher rank (\Cref{higherseveri}):
\begin{theorem*}
For any invertible sheaf $\cl$ on $\S$, we have
\begin{align*}
2\int_{\Quot^{2}_{\S}(\ce)} s_{2r+2}(\cl^{[2]}) &= -\frac{r+2}{2}\Delta(\ce) +  s_{2}(\ce\otimes\cl)^{2}-(2r+3)(r+1)s_{2}(\ce\otimes\cl) \\
&+ \frac{1}{6}(2r+3)(r+2)(r+1)s_{1}(\ce\otimes\cl)s_{1}(\S)-\binom{r+3}{4}s_{2}(\S).
\end{align*}
\end{theorem*}
The particular nature of this formula, featuring the Segre classes of $\ce\otimes\cl$ and the discriminant $\Delta(\ce) = c_{2}(\SEnd(\ce))$ 
(which is well-known in the study of stable sheaves \cite{Bog, HuL}), reflects property \ref{two}. 
The proof uses our modular resolution of singularities of $\Quot^{2}_{\S}(\ce)$ and an explicit computation on the Hilbert scheme $\bp(\ce)^{[2]}$.

In the final \SS 5, we study the cohomology of sheaves on $\Quot^{l}_{\S}(\ce)$. 
The cohomology of the structure sheaf of $\S^{[l]}$ was first studied by Fogarty \cite{Fog}. 
For $r\geqslant 2$, we observe (\Cref{coh}) that if the singularities of $\Quot^{l}_{\S}(\ce)$ are rational, 
then
\begin{equation*}
\H^{\ast}(\Quot^{l}_{\S}(\ce), \co_{\Quot^{l}_{\S}(\ce)})\xrightarrow{\sim}\S^{l}\H^{\ast}(\S, \co_{\S}).
\end{equation*}
(This is indicative of the importance of \Cref{singularities}.)
As \Cref{singularities} holds for $l=2$, the above isomorphism holds in this case as well. 
We then study the cohomology of tautological sheaves. 
Again, this is a question which is well-studied in the rank $r=1$ case \cite{Arb, EGL, Kru, Sca}.  
We conjecture (\Cref{tautcoh}) that for a locally free sheaf $\cf$ on $\S$, there is an isomorphism 
\begin{equation*}
\H^{\ast}(\Quot^{l}_{\S}(\ce), \cf^{[l]})\xrightarrow{\sim}\H^{\ast}(\S, \ce\otimes\cf)\otimes\S^{l-1}\H^{\ast}(\S, \co_{\S}).
\end{equation*}
Using \Cref{lequals2}, we show that this conjecture holds in the first non-trivial case $l=2$ (\Cref{tautcoh2}). 
By considering the Koszul resolution associated to the section $s^{[l]}$ induced by a locally free quotient $\ce\rightarrow\ce^{''}$ on $\S$, 
we obtain a spectral sequence (\Cref{spectrals}), 
relating the cohomology of Schur functors of tautological sheaves on $\Quot^{l}_{\S}(\ce)$ to the cohomology of 
Schur functors of tautological sheaves on $\Quot^{l}_{\S}(\ce^{''})$.

\subsection*{Conventions}

We only consider seperated complex schemes of finite type, and only coherent sheaves. For any sheaf $\cf$ on a scheme $\S$, we use the Grothendieck convention for the projectivisation $\bp(\cf)$; if $\cf$ is locally free, we view the pullback map $\A_{\ast}(\S)\rightarrow\A_{\ast}(\bp(\cf))$ on Chow groups as an inclusion. 
Given sheaves $\ce$ and $\cf$ on a scheme $\S$, we write
\begin{equation*}
\chi(\S, \ce, \cf) = \sum_{i} (-1)^{i} \dim\Ext^{i}_{\S}(\ce, \cf)
\end{equation*}
for the Euler pairing, and in particular $\chi(\S,\cf)=\chi(\S, \co, \cf)$ for the Euler characteristic. Throughout this paper, we assume (unless stated otherwise) that $\ce$ is a locally free sheaf of rank $r$ on a smooth projective surface $\S$.

\section{Basic properties}\label{sec:basicp}

\subsection{Quot schemes}

Let $\ce$ be a coherent sheaf on a quasi-projective scheme $\S$, and $l$ a positive integer. 
We consider the Quot scheme $\Quot^{l}_{\S}(\ce)$ of zero-dimensional length $l$ quotients of $\ce$, 
introduced by Grothendieck in his foundational paper \cite{Gro}. 
Grothendieck proved that $\Quot^{l}_{\S}(\ce)$ exists as a fine moduli scheme; 
he showed that the scheme $\Quot^{l}_{\S}(\ce)$ is quasi-projective, and that it is projective if $\S$ is projective. 
The functor of points $h_{\Quot^{l}_{\S}(\ce)}$ can be described as follows: a morphism of schemes 
\begin{equation*}
f:\T\rightarrow\Quot^{l}_{\S}(\ce)
\end{equation*}
corresponds to a quotient 
\begin{equation*}
\ce\rightarrow\cq_{f} \quad \mathrm{on} \quad \S\times\T,
\end{equation*}
where $\cq_{f}$ is a flat family of length $l$ sheaves on $\S$ parametrised by $\T$; we denote by $\cs_{f}$ the kernel of $\ce\rightarrow\cq_{f}$. In particular, a point $q$ of $\Quot^{l}_{\S}(\ce)$ corresponds to a quotient $\ce\rightarrow\cq_{q}$ on $\S$, satisfying 
\begin{equation}
\dim\cq_{q}=0 \quad \mathrm{and} \quad h^{0}(\cq_{q})=l.
\end{equation}
If $g:\T^{'}\rightarrow\T$ is a morphism of schemes, the map
\begin{equation*}
h_{\Quot^{l}_{\S}(\ce)}(\T)\rightarrow h_{\Quot^{l}_{\S}(\ce)}(\T^{'})
\end{equation*}
takes the quotient $\ce\rightarrow\cq_{f}$ on $\S\times\T$ to its pullback along $1\times g:\S\times\T^{'}\rightarrow\S\times\T$. 
As $\cq_{f}$ and $\cs_{f}$ are flat over $\T$, we have the vanishing
\begin{equation}\label{eq:vanishing}
\L_{n}(1\times g)^{\ast}\cq_{f}=0 \quad \mathrm{and} \quad \L_{n}(1\times g)^{\ast}\cs_{f}=0 \quad (n\geqslant 1).
\end{equation}
We write $\cs$ for the kernel of the universal quotient
\begin{equation}\label{eq:universal}
\ce\rightarrow\cq
\end{equation}
on $\S\times\Quot^{l}_{\S}(\ce)$. If $f:\T\rightarrow\Quot^{l}_{\S}(\ce)$ is a morphism of schemes, then by definition
\begin{equation*}
\cq_{f}=(1\times f)^{\ast}\cq \quad \mathrm{and} \quad \cs_{f}=(1\times f)^{\ast}\cs.
\end{equation*}
We will frequently use the notation
\begin{equation*}
\pi:\S\times\Quot^{l}_{\S}(\ce)\rightarrow\Quot^{l}_{\S}(\ce)
\end{equation*}
for the projection to $\Quot^{l}_{\S}(\ce)$. 

The schemes $\Quot^{l}_{\S}(\ce)$ have the following fundamental properties \cite{Gro}.
\begin{proposition}\label{I-III}
(i)  A quotient $\ce\rightarrow\ce^{''}$ on $\S$ induces a closed embedding
\begin{equation*}
\Quot^{l}_{\S}(\ce^{''})\rightarrow\Quot^{l}_{\S}(\ce).
\end{equation*}
(ii) An invertible sheaf $\cl$ on $\S$ induces an isomorphism
\begin{equation*}
\Quot^{l}_{\S}(\ce)\xrightarrow{\sim} \Quot^{l}_{\S}(\ce\otimes\cl).
\end{equation*}
(iii) The automorphism group $\Aut(\ce)$ acts naturally on $\Quot^{l}_{\S}(\ce)$.
\end{proposition}

Moreover, as observed by Grothendieck \cite{Gro}, the schemes $\Quot^{l}_{\S}(\ce)$ admit a canonical morphism to the the symmetric product  $\S^{(l)}=\S^{l}/\mathfrak{S}_{l}$. We view 
$\S^{(l)}$ as the Chow variety of effective $0$-cycles of degree $l$, with its standard stratification
\begin{equation*}
\S^{(l)}=\coprod_{\lambda}\S^{(l)}_{\lambda}
\end{equation*}
indexed by the partitions $\lambda$ of $l$. Grothendieck's morphism 
\begin{equation}\label{eq:Grothendieck}
\mu:\Quot^l_{\S}(\ce)\rightarrow\S^{(l)}
\end{equation}
takes a point $q$ to the zero-cycle $\mu(q)$ defined by $\cq_{q}$. We can also view $\S^{(l)}$ as the moduli space of zero dimensional sheaves of length $l$, and $\mu$ as the forgetful morphism taking $q$ to $\cq_{q}$ \cite{HuL}. We will sometimes use the notation
\begin{equation*}
\Quot^{l}_{\S}(\ce)_{\lambda}=\mu^{-1}(\S^{(l)}_{\lambda}).
\end{equation*}

For future reference, we include the following auxiliary result. 
\begin{lemma}\label{locfree}
If $\ce$ is locally free of rank $r$, then $\Quot^{l}_{\S}(\ce)$ is locally isomorphic to $\Quot^{l}_{\S}(\co^{\oplus r})$.
\end{lemma}

\begin{proof}
It suffices to show that for any point  $q$ of $\Quot^l_{\S}(\ce)$, 
there is an open neighbourhood $\U\subset\S$ of the support of $\cq_{q}$ such that the restriction $\ce|_\U$ is trivial. 
As the support of $\cq_{q}$ is finite, we can find a hyperplane section $\H$ of $\S$ which is disjoint from it. 
The complement $\S-\H$ is affine; as locally free sheaves (of constant rank) over semi-local schemes are trivial (Lemma 1.4.4 of \cite{BrH}), 
we can obtain $\U$ by semi-localising  $\S-\H$ with respect to the support of $\cq_{q}$.
\end{proof}

If $\S$ is irreducible of dimension $\geqslant 3$, then it is well-known \cite{Iar} that the Hilbert scheme of points
\begin{equation*}
\S^{[l]} = \Quot^{l}_{\S}(\co)
\end{equation*}
is reducible for $l\gg 0$; in fact, $\S^{[l]}$ can even be non-reduced \cite{Jel}. For this reason, we will assume, in all that follows, that $\ce$ is a locally free sheaf of rank $r$ on a smooth projective surface $\S$. In this case,  $\Quot^{l}_{\S}(\ce)$ has the following basic properties.
\begin{theorem}\label{basic}
Let $\ce$ be a locally free sheaf of rank $r$ on a smooth surface $\S$. \\
(i) If $\S$ is irreducible, then $\Quot^{l}_{\S}(\ce)$ is irreducible of dimension $l(r+1)$. \\
(ii) The open subscheme $\Quot^{l}_{\S}(\ce)_{(1^{l})}$ is dense in $\Quot^{l}_{\S}(\ce)$. \\
(iii) The scheme $\Quot^{l}_{\S}(\ce)$ is birational to the symmetric product $\bp(\ce)^{(l)}$.
\end{theorem}
Part (i) is Corollary 2.6 of \cite{Reg}, while (ii) is a result of Gieseker and Li \cite{GiL}. As for (iii), there is a canonical isomorphism $\Quot^{1}_{\S}(\ce)\xrightarrow{\sim}\bp(\ce)$ by \cite{Kle} (see also \Cref{lequals1}), and one can consider the morphism of schemes
\begin{equation*}
\Quot^{1}_{\S}(\ce)^{(l)}\rightarrow\S^{(l)},
\end{equation*}
induced by the the $l$-fold product of (\ref{eq:Grothendieck}). Writing $\Quot^{1}_{\S}(\ce)^{(l)}_{(1^{l})}$ for the preimage of $\S^{(l)}_{(1^{l})}$ under the latter morphism, (iii) comes down to the observation that the open subscheme $\Quot^{l}_{\S}(\ce)_{(1^{l})}$ of $\Quot^{1}_{\S}(\ce)^{(l)}$ is isomorphic to $\Quot^{1}_{\S}(\ce)^{(l)}_{(1^{l})}$.
\begin{remark}
If $\S_{1}, \ldots, \S_{n}$ are the connected components of $\S$, then
\begin{equation*}
\Quot^{l}_{\S}(\ce)=\coprod_{l_{1}+\cdots+l_{n}=l} \prod_{i=1}^{n}\Quot^{l_{i}}_{\S_{i}}(\ce\vert_{\S_{i}}).
\end{equation*}
By \Cref{basic} (i), this is the decomposition of $\Quot^{l}_{\S}(\ce)$ into irreducible components. 
In particular, the Quot scheme $\Quot^{l}_{\S}(\ce)$ is of pure dimension $l(r+1)$.
\end{remark}

A particularly nice, inductive proof of part (i) of \Cref{basic} was given by Ellingsrud and Lehn \cite{ElL}. 
The induction step is obtained by considering the flag Quot scheme 
\begin{equation}\label{eq:flag}
	\begin{tikzcd}
		& \Quot^{l, l+1}_{\S}(\ce) \arrow[rd] \arrow[ld] &  \\
		\S\times\Quot^{l}_{\S}(\ce) & & \Quot^{l+1}_{\S}(\ce),
	\end{tikzcd}
\end{equation}
where the morphism
\begin{equation*}
\Quot^{l,l+1}_{\S}(\ce) \rightarrow \S\times\Quot^{l}_{\S}(\ce)
\end{equation*}
can be identified with the projectivisation of the universal subsheaf $\cs$ on $\S\times\Quot^{l}_{\S}(\ce)$; 
Ellingsrud and Lehn prove that $\bp(\cs)$ is irreducible.

\begin{remark}\label{aux}
For any point $q$ of $\Quot^{l}_{\S}(\ce)$, the sheaf $\cs_{q}$ is torsion-free, 
and $\ce$ is its double dual. Conversely, by considering the double dual, 
we see that every torsion-free sheaf on $\S$ is of the form  $\cs_{q}$ for some point $q$ of $\Quot^{l}_{\S}(\ce)$ for some $\ce$ and $l$. 
\end{remark}
\subsection{Tautological sheaves}

Let $\cf$ be a locally free sheaf on $\S$. The tautological sheaf on $\Quot^{l}_{\S}(\ce)$ induced by $\cf$ is defined by 
\begin{equation*}
\cf^{[l]}=\pi_{\ast}(\cq\otimes\cf).
\end{equation*}
These sheaves are intrinsic to $\Quot^{l}_{\S}(\ce)$. For us, a key property is their compatibility with base change. 
Consider a morphism of schemes $f:\T\rightarrow\Quot^{l}_{\S}(\ce)$ and the induced Cartesian diagram
\begin{equation*}
	\begin{tikzcd}
		\S\times\T \arrow[r] \arrow[d, "\pi_{\T}", swap] & \S\times\Quot^l_{\S}(\ce) \arrow[d, "\pi"]   \\
		\T \arrow[r] & \Quot^l_{\S}(\ce).
	\end{tikzcd}
\end{equation*}

\begin{lemma}\label{basechange}
For every morphism of schemes $f:\T\rightarrow\Quot^{l}_{\S}(\ce)$, the base change map
\begin{equation*}
f^{\ast}\cf^{[l]}\rightarrow\pi_{\T\ast}(\cq_{f}\otimes\cf)
\end{equation*}
is an isomorphism, and the sheaf $\cf^{[l]}$ is locally free.
\end{lemma}
\begin{proof}
For every point $q$ of $\Quot^{l}_{\S}(\ce)$, $\dim\cq_{q}=0$ implies the vanishing
\begin{equation*}
\H^{1}(\S, \cq_{q}\otimes\cf)=0.
\end{equation*}
As $\cf$ is locally free, the sheaf $\cf\otimes\cq$ on $\S\times\Quot^l_{\S}(\ce)$ is flat over $\Quot^l_{\S}(\ce)$, and the conclusion follows from Grothendieck's theorem on cohomology and base change.
\end{proof}
\begin{corollary}
The fibre of $\cf^{[l]}$ over a point $q$ is
\begin{equation*}
\cf^{[l]}(q)=\H^0(\S, \cq_{q}\otimes\cf).
\end{equation*}
\end{corollary}

As a consequence of \Cref{basechange}, we obtain the compatibility of tautological sheaves with the three properties (i)-(iii) of \Cref{I-III}.

\begin{proposition}\label{tautcomp}
(i) If $\ce\rightarrow\ce^{''}$ is a quotient on $\S$, then the tautological sheaf $\cf^{[l]}$ on $\Quot^{l}_{\S}(\ce)$ restricts to the tautological sheaf $\cf^{[l]}$ on $\Quot^{l}_{\S}(\ce^{''})\subset\Quot^{l}_{\S}(\ce)$. \\
(ii) If $\cl$ is an invertible sheaf on $\S$, then the tautological sheaf $\cf^{[l]}$ on $\Quot^l_{\S}(\ce\otimes\cl)$ corresponds to the tautological sheaf $(\cf\otimes\cl)^{[l]}$ on $\Quot^l_{\S}(\ce)$. \\
(iii) The sheaf $\cf^{[l]}$ admits a canonical $\Aut(\ce)$-linearisation.
\end{proposition}

\begin{proof}
Parts (i) and (ii) are immediate consequences of \Cref{basechange}. As for (iii), consider the morphism given by the action of $\Aut(\ce)$
\begin{equation*}
\sigma:\Aut(\ce)\times\Quot^{l}_{\S}(\ce)\rightarrow\Quot^{l}_{\S}(\ce),
\end{equation*} 
and the projection $\tau$ to the second factor. Applying \Cref{basechange} to both maps yields, by composition of the corresponding base change maps, an isomorphism 
\begin{equation*}
\sigma^{\ast}\cf^{[l]}\xrightarrow{\sim}\tau^{\ast}\cf^{[l]}
\end{equation*}
of sheaves on $\Aut(\ce)\times\Quot^{l}_{\S}(\ce)$. It is tedious, but straightforward to show that this isomorphism satisfies the cocycle condition.
\end{proof}

By taking the pushforward $\pi_{\ast}$ of the universal quotient (\ref{eq:universal}), we obtain a morphism of coherent sheaves
\begin{equation*}
\H^{0}(\S, \ce\otimes\cf)\otimes\co_{\Quot^{l}_{\S}(\ce)}\rightarrow\cf^{[l]}
\end{equation*}
on $\Quot^{l}_{\S}(\ce)$, and in particular a map on global sections
\begin{equation}\label{eq:global}
\H^{0}(\S, \ce\otimes\cf)\rightarrow\H^{0}(\Quot^{l}_{\S}(\ce), \cf^{[l]}).
\end{equation}
We denote by $s^{[l]}$ the image of $s\in\H^{0}(\S, \ce\otimes\cf)$ under the latter map; the value of $s^{[l]}$ at a given point $q$ of $\Quot^{l}_{\S}(\ce)$ is given by
\begin{equation*}
\co\xrightarrow{s}\ce\otimes\cf\rightarrow\cq_{q}\otimes\cf.
\end{equation*}
The map (\ref{eq:global}) is compatible with \Cref{I-III} (i) in the following sense. 

\begin{lemma}\label{sectiondiag}
For every quotient $\ce\rightarrow\ce^{''}$ on $\S$, the diagram
\begin{equation*}
	\begin{tikzcd}
		\H^{0}(\S,\ce\otimes\cf) \arrow[r] \arrow[d] & \H^{0}(\Quot^{l}_{\S}(\ce), \cf^{[l]}) \arrow[d, "\iota^{\ast}"]  \\
		\H^{0}(\S,\ce^{''}\otimes\cf) \arrow[r] & \H^{0}(\Quot^{l}_{\S}(\ce^{''}), \cf^{[l]})
	\end{tikzcd}
\end{equation*}
is commutative.
\end{lemma}

\begin{proof}
On $\S\times\Quot^{l}_{\S}(\ce)$ we have a commutative diagram
\begin{equation}\label{eq:comm1}
	\begin{tikzcd}
		\ce\otimes\cf \arrow[r] \arrow[d] & \cq\otimes\cf \arrow[d]  \\
		(1\times\iota)_{\ast}(1\times\iota)^{\ast}(\ce\otimes\cf) \arrow[r] & (1\times\iota)_{\ast}(1\times\iota)^{\ast}(\cq\otimes\cf),
	\end{tikzcd}
\end{equation}
where the horizontal morphisms are induced by the universal quotient (\ref{eq:universal}), and the vertical morphisms are given by restriction. On the other hand, by definition of $\iota$, we have a commutative diagram
\begin{equation}\label{eq:comm2}
\begin{tikzcd}
	(1\times\iota)^{\ast}(\ce\otimes\cf) \arrow[r] \arrow[d] & (1\times\iota)^{\ast}(\cq\otimes\cf) \arrow[d, "\isom"]  \\
	(1\times\iota)^{\ast}(\ce^{''}\otimes\cf) \arrow[r] & \cq^{''}\otimes\cf
\end{tikzcd}
\end{equation}
on $\S\times\Quot^{l}_{\S}(\ce^{''})$, where the horizontal maps are induced by the universal quotients associated to $\Quot^{l}_{\S}(\ce)$ and $\Quot^{l}_{\S}(\ce^{''})$, respectively. It remains to apply $(1\times\iota)_{\ast}$ to (\ref{eq:comm2}), combine the resulting diagram with (\ref{eq:comm1}), and take global sections.
\end{proof}

The identification of $\Quot^1_{\S}(\ce)$ with the projectivisation $\bp(\ce)$ of $\ce$ is well-known \cite{Kle}; we will need a more detailed description of the identification. Let $p:\bp(\ce)\rightarrow\S$ be the projection, and $\co(1)$ the canonical invertible sheaf on $\bp(\ce)$.
\begin{proposition}\label{lequals1}
(i) There exists a morphism of schemes
$$\rho^{1}:\bp(\ce)\rightarrow\Quot^1_{\S}(\ce)$$
taking a point $z$ of $\bp(\ce)$ to the quotient $\rho^{1}(z)$ given by $\ce=p_{\ast} \co(1)\rightarrow p_{\ast} \co_{z}(1)$. \\
(ii) The morphism $\rho^{1}$ is an isomorphism. \\
(iii) For any locally free sheaf $\cf$ on $\S$,
\begin{equation*}
\rho^{1\ast}\cf^{[1]}=\cf(1).
\end{equation*}
\end{proposition}

\begin{proof}
(i) Let $\Delta\subset\bp(\ce)\times\bp(\ce)$ be the diagonal, and consider the diagram
\begin{equation}\label{eq:diag1}
	\begin{tikzcd}
		\bp(\ce) \arrow[r] \arrow[d] & \bp(\ce)\times\bp(\ce)  \arrow[d, "p\times 1"]  \\
		\S \arrow[r] & \S\times\bp(\ce),
	\end{tikzcd}
\end{equation}
Taking the pushforward $(p\times 1)_{\ast}$ of the quotient $\co(1)\rightarrow\co_{\Delta}(1)$ on $\bp(\ce)\times\bp(\ce)$ yields a morphism of coherent sheaves 
\begin{equation}\label{eq:family1}
\ce\rightarrow(p\times 1)_{\ast}\co_{\Delta}(1)
\end{equation}
on $\S\times\bp(\ce)$. Since the map $\Delta\rightarrow\S\times\bp(\ce)$ induced by $p\times 1$ is affine, (\ref{eq:family1}) restricts to 
\begin{equation*}
\rho^{1}(z):\ce=p_{\ast} \co(1)\rightarrow p_{\ast}\co_{z}(1)=\co_{p(z)} \quad \mathrm{on \ each} \quad \S\times\{z\}.
\end{equation*}
Once we have shown that each $\rho^{1}(z)$ is surjective, we can define the morphism $\rho^1$ to be the one corresponding to the family of quotients (\ref{eq:family1}). It suffices to show that $\rho^{1}(z)$ is nonzero, and of course this is the case if and only if $\rho^{1}(z)\otimes\cl$ is nonzero for some invertible sheaf $\cl$. Under the canonical isomorphism
\begin{equation*}
\bp(\ce)\xrightarrow{\sim}\bp(\ce\otimes\cl),
\end{equation*}
the morphism $\rho^{1}(z)\otimes\cl$ can be identified with the pushforward of
\begin{equation*}
\co_{\bp(\ce\otimes\cl)}(1)\rightarrow\co_{\bp(\ce\otimes\cl)}(1)\otimes\co_{z},
\end{equation*}
and we can therefore assume that $\ce$ is globally generated. The sheaf $\co(1)$ is then also globally generated, and in particular, the evaluation map
\begin{equation*}
\H^0(\co(1))\rightarrow\H^0(\co_{z})
\end{equation*}
is nonzero, and $\rho^{1}(z)$ is therefore nonzero as well.

(ii) It is clear that $\rho^1$ induces a bijection on points; as $\Quot^{1}_{\S}(\ce)$ is smooth (see \Cref{tangent}), 
Zariski's main theorem implies that $\rho^{1}$ is an isomorphism.

(iii) Applying \Cref{basechange} to $\rho^{1}$, 
we obtain that the pullback $\rho^{1\ast}\cf^{[1]}$ is given by the pushforward of $\cf\otimes(p\times 1)_{\ast}\co_{\Delta}(1)$ to $\bp(\ce)$, 
which is $\cf(1)$ by the projection formula and (\ref{eq:diag1}).
\end{proof}
\subsection{Obstruction theory}\label{pot}

By the work of Grothendieck \cite{Gro}, the Quot scheme $\Quot^l_{\S}(\ce)$ has an intrinsic obstruction theory. 
The tangent space of $\Quot^l_{\S}(\ce)$ at a point $q$ can be identified with $\Hom(\cs_{q},\cq_{q})$, 
while the obstruction space is $\Ext^{1}(\cs_{q},\cq_{q})$. 
The obstruction theory of $\Quot^l_{\S}(\ce)$ can be viewed as a morphism in the derived category
\begin{equation}\label{eq:at}
\mathrm{at}:\R\SHom_{\pi}(\cs,\cq)^{\vee}\rightarrow\L_{\Quot^l_{\S}(\ce)},
\end{equation}
where $\L_{\Quot^l_{\S}(\ce)}$ denotes the cotangent complex of $\Quot^l_{\S}(\ce)$. 
While this morphism and its implications are well-known, we shall give a brief account; 
we stress that all the ideas presented here are due to Grothendieck \cite{Gro} and especially Illusie \cite{Ill}. 
First of all, Grothendieck duality for $\pi$ gives a canonical isomorphism
\begin{equation*}
\R\SHom_{\pi}(\cs,\cq)^{\vee} \xrightarrow{\sim}\R\SHom_{\pi}(\cq,\cs\otimes\omega_{\pi}[2]).
\end{equation*}
Using this isomorphism and Grothendieck duality, we obtain a functorial isomorphism
\begin{equation}\label{eq:func}
\Ext^{0}_{\S\times\Quot^{l}_{\S}(\ce)}(\R\SHom_{\pi}(\cs,\cq)^{\vee}, -) \xrightarrow{\sim} \Ext^{0}_{\S\times\Quot^{l}_{\S}(\ce)}(\cs, \pi^{\ast}(-)\otimes^{\L}\cq).
\end{equation}
Consider the Atiyah class
\begin{equation*}
\mathrm{At}\in\Ext^{1}_{\S\times\Quot^{l}_{\S}(\ce)}(\cq, \pi^{\ast}\L_{\Quot^{l}_{\S}(\ce)}\otimes^{\L}\cq)
\end{equation*}
of $\cq$ relative to the projection $\S\times\Quot^{l}_{\S}(\ce)\rightarrow\S$. One has a canonical factorisation
\begin{equation*}
\mathrm{At} = \mathrm{at}[1]\circ\delta,
\end{equation*}
where
\begin{equation*}
\mathrm{at}\in\Ext^{0}_{\S\times\Quot^{l}_{\S}(\ce)}(\cs, \pi^{\ast}\L_{\Quot^{l}_{\S}(\ce)}\otimes^{\L}\cq), \quad \mathrm{and} \quad \delta\in\Ext^{1}_{\S\times\Quot^{l}_{\S}(\ce)}(\cq, \cs)
\end{equation*}
is the class of the universal exact sequence on $\S\times\Quot^{l}_{\S}(\ce)$. 
(For a precise definition of $\mathrm{at}$ we refer to \cite{Ill} or \cite{Gil}.) 
Under the isomorphism (\ref{eq:func}), $\mathrm{at}$ yields (\ref{eq:at}). 
\begin{theorem}\label{ot}
The morphism $\mathrm{at}$ is an obstruction theory in the sense of \cite{BeF}.
\end{theorem}
The idea behind this theorem is as follows. Consider a morphism $f:\T\rightarrow\Quot^{l}_{\S}(\ce)$, and closed embedding $\T\rightarrow\overline{\T}$ with ideal sheaf $\ci$ satisfying $\ci^{2}=0$. The composite 
\begin{equation*}
\L f^{\ast}\L_{\Quot^{l}_{\S}(\ce)}\rightarrow\L_{\T}\rightarrow\ci[1],
\end{equation*}
where $\L_{\T}\rightarrow\ci[1]$ is given by truncation of $\L_{\T}\rightarrow\L_{\T/\overline{\T}}$, defines a class
\begin{equation*}
\omega(f)\in\Ext^{1}(\L f^{\ast}\L_{\Quot^{l}_{\S}(\ce)}, \ci).
\end{equation*}
By a theorem of Grothendieck \cite{Gro3}, the morphism $f$ extends to a morphism $\overline{\T}\rightarrow\Quot^{l}_{\S}(\ce)$ if and only if $\omega(f)=0$. The meaning of \Cref{ot} is essentially that
\begin{equation*}
\omega(f)\circ\L f^{\ast}\mathrm{at}=0
\end{equation*}
if and only if the morphism $f$ extends to a morphism $\overline{\T}\rightarrow\Quot^{l}_{\S}(\ce)$. Parallel to (\ref{eq:func}), one finds, by using base change, Grothendieck duality, and (\ref{eq:vanishing}), a canonical isomorphism
\begin{equation}\label{eq:funcf}
\Ext^{1}_{\T}(\L f^{\ast}\R\SHom_{\pi}(\cs, \cq)^{\vee}, \ci)\xrightarrow{\sim} \Ext^{1}_{\S\times\T}(\cs_{f}, \pi_{\T}^{\ast}\ci\otimes\cq_{f}),
\end{equation}
where $\pi_{\T}:\S\times\T\rightarrow\T$ is the projection. By the theory of Illusie \cite{Ill}, there is a class in $\Ext^{1}_{\S\times\T}(\cs_{f}, \pi_{\T}^{\ast}\ci\otimes\cq_{f})$ which is the obstruction to extending the quotient $\ce\rightarrow\cq_{f}$ on $\S\times\T$ to a quotient on $\S\times\overline{\T}$; since extending $\ce\rightarrow\cq_{f}$ is equivalent to extending $f$, \Cref{ot} comes down to the observation that under (\ref{eq:funcf}), $\omega(f)\circ\L f^{\ast}\mathrm{at}$ corresponds to Illusie's class.
\begin{remark}
\Cref{ot} holds of course for more general Quot schemes, but if $\S$ is a surface, the complex $\R\SHom_{\pi}(\cs,\cq)^{\vee}$ is perfect of amplitude contained in $[-1, 0]$. Indeed, consider
\begin{equation}\label{eq:Tvir}
\T^{\vir}_{\Quot^{l}_{\S}(\ce)}=\R\SHom_{\pi}(\cs,\cq).
\end{equation}
Since $\cs$ is a family of torsion-free sheaves on $\S$, $\cs$ is quasi-isomorphic to a two-term complex $\cf_{1}\rightarrow\cf_{0}$ of locally free sheaves.
 By cohomology and base change, any locally free sheaf $\cf$ on $\S\times\Quot^{l}_{\S}(\ce)$ is $\R\SHom_{\pi}(-,\cq)$-acyclic, 
 and $\SHom_{\pi}(\cf,\cq)$ is locally free. In particular, (\ref{eq:Tvir}) is given by the two-term complex
\begin{equation*}
	\SHom_{\pi}(\cf_{0},\cq)\rightarrow\SHom_{\pi}(\cf_{1},\cq)
\end{equation*}
of locally free sheaves.	
\end{remark}
It follows from this remark and \Cref{ot} that $\mathrm{at}$ defines a perfect obstruction theory of virtual dimension $lr$. 
The general construction of \cite{BeF} yields a virtual class
\begin{equation*}
[\Quot^{l}_{\S}(\ce)]^{\vir} \in \A_{lr}(\Quot^{l}_{\S}(\ce)).
\end{equation*}
(Virtual classes for Quot schemes were first considered in the papers \cite{CiK, MaO} and \cite{MOP1}.) 
As $\Quot^{l}_{\S}(\ce)$ is a projective scheme, this class is given by the formula
\begin{equation}\label{eq:Siebert}
[\Quot^{l}_{\S}(\ce)]^{\vir}=\left\lbrace s(\T^{\vir}_{\Quot^{l}_{\S}(\ce)})\cap c_{\F}(\Quot^{l}_{\S}(\ce))\right\rbrace_{lr}
\end{equation}
of Siebert \cite{Sie}. Here $c_{\F}(\Quot^{l}_{\S}(\ce))$ is the Fulton class, 
which may be viewed as the Chern class of the singular scheme $\Quot^{l}_{\S}(\ce)$. 
If $\Quot^{l}_{\S}(\ce)$ is a closed subscheme of a smooth scheme $\A$, then $c_{\F}(\Quot^{l}_{\S}(\ce))$ is given by
\begin{equation*}
c_{\F}(\Quot^{l}_{\S}(\ce)) = c(\Theta_{\A}|_{\Quot^{l}_{\S}(\ce)})\cap s(\C_{\Quot^{l}_{\S}(\ce) / \A}),
\end{equation*}
where $\Theta_{\A}$ is the tangent sheaf of $\A$, 
and $s(\C_{\Quot^{l}_{\S}(\ce) / \A})$ is the Segre class of the normal cone 
$\C_{\Quot^{l}_{\S}(\ce) / \A}$ of $\Quot^{l}_{\S}(\ce)\subset\A$ (see \cite{Ful}, Example 4.2.6).

\section{Local properties}\label{sec:localp}

\subsection{Embedding dimension}

We start with the following vanishing result.
\begin{lemma}\label{chi}
For every point $q$ of $\Quot^l_{\S}(\ce)$, we have
\begin{equation*}
\chi(\S, \cq_q,\cq_q)=0.
\end{equation*}
\end{lemma}

\begin{proof}
Applying Theorem 3 of \cite{Ban} to the universal quotient and the pullback of $\cq_q$ along the projection $\S\times\Quot^l_{\S}(\ce)\rightarrow\S$, 
we obtain that the function
\begin{equation*}
\Quot^l_{\S}(\ce)\rightarrow\bz, \quad q^{'}\mapsto\chi(\S, \cq_{q},\cq_{q^{'}})
\end{equation*}
is locally constant. 
By picking a point $q^{'}$ in the same component of $\Quot^l_{\S}(\ce)$ as $q$ such that the supports of $\cq_q$ and $\cq_{q'}$ are disjoint, we have
\begin{equation*}
\SExt^{n}_{\S}(\cq_{q},\cq_{q^{'}}) = 0
\end{equation*}
for all $n\geqslant 0$, and in particular
\begin{equation*}
\chi(\S, \cq_q,\cq_q) = \chi(\S, \cq_{q},\cq_{q'}) = 0. \qedhere
\end{equation*}
\end{proof}

\begin{theorem}\label{tangent}
For every point $q$ of $\Quot^{l}_{\S}(\ce)$,
\begin{equation*}
\dim\T_{q} \Quot^l_{\S}(\ce)=lr+h^{0}(\S, \SEnd(\cq_{q})).
\end{equation*}
\end{theorem}

\begin{proof}
As $\T_{q} \Quot^l_{\S}(\ce)=\Hom_{\S}(\cs_{q}, \cq_{q})$, the short exact sequence
\begin{equation*}
0\rightarrow\cs_{q}\rightarrow\ce\rightarrow\cq_{q}\rightarrow 0
\end{equation*}
induces an exact sequence of Ext groups
\begin{equation}\label{eq:les}
0\rightarrow\H^{0}(\S, \SEnd(\cq_{q}))\rightarrow\Hom_{\S}(\ce,\cq_q)\rightarrow\T_{q} \Quot^l_{\S}(\ce)\rightarrow\Ext^1_{\S}(\cq_q,\cq_q)\rightarrow 0.
\end{equation}
Here we have used the dimensional vanishing
\begin{equation*}
\Ext^n_{\S}(\ce,\cq_{q})\xrightarrow{\sim}\H^{n}(\S, \ce^{\vee}\otimes\cq_{q})=0 \quad (n\geqslant 1).
\end{equation*}
Moreover, $\Hom_{\S}(\ce,\cq_q)\xrightarrow{\sim}\H^{0}(\S, \ce^{\vee}\otimes\cq_{q})$ has dimension $lr$. By Serre duality
\begin{equation*}
\Ext^2_{\S}(\cq_q,\cq_q)\simeq\Hom_{\S}(\cq_q,\cq_q\otimes\omega_{\S})^{\vee}\simeq\H^{0}(\S, \SEnd(\cq_{q}))^{\vee}.
\end{equation*}
\Cref{chi} now implies the equality
\begin{equation*}
\dim\Ext^1_{\S}(\cq_{q},\cq_{q}) = 2h^{0}(\S, \SEnd(\cq_{q})),
\end{equation*}
and it remains to take dimensions in (\ref{eq:les}).
\end{proof}

By combining \Cref{basic} and \Cref{tangent}, we get the following smoothness criterion:
\begin{corollary}\label{smooth}
The scheme $\Quot^l_{\S}(\ce)$ is smooth at $q$ if and only if $h^{0}(\S, \SEnd(\cq_{q}))=l$. 
\end{corollary}
In particular, this implies the follwing basic result.
\begin{corollary}\label{sing}
The scheme $\Quot^l_{\S}(\ce)$ is generically smooth; it is singular for $l,r\geqslant 2$.
\end{corollary}
Indeed, $\Quot^l_{\S}(\ce)_{(1^{l})}$ is contained in the smooth locus $\Quot^{l}_{\S}(\ce)^{\circ}$ of $\Quot^{l}_{\S}(\ce)$; 
for the second part, it is easy to construct a point $q$ such that $\cq_{q}$ has a direct summand of the form $\co_{s}^{\oplus 2}$ 
for some $s\in\S$ (see also \cite{Reg, OP}).

\Cref{smooth} also gives:
\begin{corollary}\label{smoothI}
For any locally free quotient $\ce\rightarrow\ce^{''}$ on $\S$, 
\begin{equation*}
\Quot^{l}_{\S}(\ce^{''})^{\circ}=\Quot^{l}_{\S}(\ce)^{\circ}\cap\Quot^{l}_{\S}(\ce^{''}).
\end{equation*}
\end{corollary}

For $l=2$ and any smooth point $q$ of $\Quot^{l}_{\S}(\ce)$, 
the corresponding sheaf $\cq_{q}$ is a structure sheaf. 
Contrary to Lemma 2.13 in the well-known paper \cite{Muk}, this is no longer true when $l\geqslant 3$. 
\begin{example}
Let $\Z\subset\S$ be a length $l$ subscheme which is not Gorenstein. 
The dualising sheaf $\omega_{\Z}$ cannot be a structure sheaf; 
as the canonical morphism $\co_{\Z}\rightarrow\SEnd(\omega_{\Z})$ is an isomorphism, 
we have $h^{0}(\S, \SEnd(\omega_{\Z}))=h^{0}(\Z, \co_{\Z})=l$. 
Taking a presentation $\co^{\oplus r}_{\Z}\rightarrow\omega_{\Z}$, 
we obtain a smooth point $q$ of $\Quot^{l}_{\S}(\co^{\oplus r})$.
\end{example}

\begin{remark}
It follows from (\ref{eq:les}) that for any point $q$ of $\Quot^l_{\S}(\ce)$ we have the upper bound
\begin{equation}\label{eq:bound}
\dim\T_{q} \Quot^l_{\S}(\ce)\leqslant 2lr.
\end{equation}
This bound is sharp. (Take $\ce=\co^{\oplus r}$ and $r\geqslant l$; 
fix a point $s$ of $\S$ and consider the quotient $q$ given by the projection $\co^{\oplus r}\rightarrow\co_s^{\oplus l}$. 
Then $\H^{0}(\S, \SEnd(\co_s^{\oplus l}))$ has dimension $l^{2}$, and $\T_{q} \Quot^l_{\S}(\ce)$ has dimension $lr+l^{2}$, 
which is $2lr$ for $l=r$.) It is a remarkable fact that in the local case of $\ce=\co^{\oplus r}$ on $\S=\ba^{2}$, 
there is a smooth scheme of dimension $2lr$ containing $\Quot^l_{\S}(\ce)$. 
Indeed, let $\ba^{2}\subset\bp_{2}$ and write $\L=\bp_2-\ba^2$. 
Let $\M(r,l)$ be the moduli space of torsion free sheaves on $\bp_2$ of rank $r$ and $c_2=l$, 
with a fixed trivialisation (framing) over $\L$; this scheme is smooth and irreducible of dimension $2lr$ \cite{Nak}. 
Viewing $\Quot^l_{\ba^{2}}(\co^{\oplus r})$ as an open subscheme of $\Quot^l_{\bp_{2}}(\co^{\oplus r})$, we get a closed embedding 
\begin{equation*}
\Quot^l_{\ba^{2}}(\co^{\oplus r}) \subset \M(r,l)
\end{equation*}
by taking a point $q$ to the torsion-free sheaf $\cs_q$ on $\bp_{2}$ with framing given by $\cs_{q}|_{\L}\xrightarrow{\sim}\co^{\oplus r}_{\L}$. 
\end{remark}
\subsection{Conjectural description of the singularities}

To have a qualitative understanding of the singularities of $\Quot^{l}_{\S}(\ce)$ is both of interest in itself and important for many applications. We pose the following conjecture:
\begin{conjecture}\label{singularities}
The scheme $\Quot^{l}_{\S}(\ce)$ has rational singularities.
\end{conjecture}
In particular, we expect that $\Quot^l_{\S}(\ce)$ is reduced. (This question was raised in \cite{Reg}.)
\begin{remark}
If $\S$ is connected, then it is clear from (\ref{eq:flag}) that $\Quot^{l}_{\S}(\ce)$ is connected ($\bp(\cs)$ is connected); as a connected normal scheme is irreducible, \Cref{singularities} in particular implies \Cref{basic} (i).
\end{remark}

To motivate \Cref{singularities}, we elaborate on Nakajima's description of $\Quot^{l}_{\ba^{2}}(\co^{\oplus r})$. Recall (e.g. \cite{Gin}) that the commuting scheme of a Lie algebra $\mathfrak{g}$ is the closed subscheme
\begin{equation*}
\C(\mathfrak{g})\subset\mathfrak{g}\times\mathfrak{g}
\end{equation*}
defined by the vanishing $[x, y]=0$ of the Lie bracket.
Consider now a vector space $\V$ of dimension $l$, and the associated general linear Lie algebra $\mathfrak{g}=\mathfrak{gl}(\V)$.
The reductive group $\GL(\V)$ acts on the affine scheme $\A=\C(\mathfrak{gl}(\V))\times\V^{r}$ by
\begin{equation*}
g(x,y,v_1,\ldots,v_r)=(gxg^{-1},gyg^{-1},gv_1,\ldots,gv_r),
\end{equation*}
and a natural $\GL(\V)$-linearisation $\chi$ of the structure sheaf $\co_{\A}$
is given by 
\begin{equation*}
g(a,t)=(ga,\det(g)^{-1}t).
\end{equation*}
Let $\U$ be the open subscheme of $\A$ whose points $a=(x,y,v_1,\ldots,v_r)$ are such that there exists no proper subspace of $\V$ containing $v_1, \ldots, v_r$ which is invariant for both $x$ and $y$. This is precisely the locus of stable points in the sense of geometric invariant theory: 
\begin{lemma}\label{GIT}
We have $\U=\A^{s}=\A^{ss}$, and the action of $\GL(\V)$ on $\U$ is free.
\end{lemma}
Mutatis mutandis, the proof is the same as in the $r=1$ case \cite{Nak}. We consider the geometric invariant theory quotient
\begin{equation*}
\A\git_{\chi}\GL(\V) = \U/\GL(\V).
\end{equation*}
By \Cref{GIT} and Luna's theorem \cite{HuL}, the morphism
\begin{equation}\label{eq:proj}
\U\rightarrow\U/\GL(\V)
\end{equation}
is a principal $\GL(\V)$-bundle. We can now state the ADHM description.
\begin{theorem}[Nakajima \cite{Nak}]\label{ADHM}
There is a canonical isomorphism
\begin{equation*}
\U/\GL(\V)\xrightarrow{\sim}\Quot^{l}_{\ba^2}(\co^{\oplus r}).
\end{equation*}
\end{theorem}
Again, the proof in the $r=1$ case (see \cite{Nak}) readily generalises to higher rank. Since (\ref{eq:proj}) is a principal $\GL(\V)$-bundle, the singularities of $\Quot^{l}_{\ba^2}(\co^{\oplus r})$ can be studied in terms of those of $\U$. One is thus led to study the singularities of the commuting schemes, a question which is addressed by the following long-standing conjecture of Artin and Hochster.
\begin{conjecture}[Artin-Hochster]\label{AH}
The commuting scheme $\C(\mathfrak{gl}(\V))$ is Cohen-Macaulay.
\end{conjecture}
However, it is not known whether $\C(\mathfrak{gl}(\V))$ is reduced or not; Ginzburg proved that its normalisation is Cohen-Macaulay \cite{Gin}. As $\ce$ is locally free and $\S$ is smooth, it suffices to prove \Cref{singularities} in the case of $\ce=\co^{\oplus r}$ on $\S=\ba^{2}$. Hence \Cref{singularities} is implied by a stronger (rational singularities) version of \Cref{AH}. We use this to show:
\begin{proposition}\label{ratsing}
The scheme $\Quot^{2}_{\S}(\ce)$ has rational singularities.
\end{proposition}

\begin{proof}
As indicated above, it suffices to show that the commuting scheme
\begin{equation*}
\C(\mathfrak{gl}(\V))\xrightarrow{\sim}\C(\mathfrak{sl}(\V))\times\ba^2.
\end{equation*}
has rational singularities when $\V$ has dimension $l=2$. Writing out the equations defining $\C(\mathfrak{sl}(\V))$, it is evident that $\C(\mathfrak{sl}(\V))$ is just the generic determinantal scheme $\M_1(3,2)$ of $3\times 2$-matrices of rank $\leqslant 1$. By the first fundamental theorem of invariant theory, it is a GIT quotient of a smooth scheme (see e.g. \cite{ACGH}). This implies that $\M_{1}(3,2)$ has rational singularities by the theorem of Boutot \cite{Bou}. (Alternatively, one can show this directly by considering the standard resolution of singularities $\tilde{\M}_{1}(3,2)\rightarrow\M_{1}(3,2)$.)
\end{proof}

\begin{remark}
(i) The scheme $\C(\mathfrak{sl}_2)$ is not Gorenstein at $0$, from which one can deduce that $\Quot^2_{\S}(\ce)$ cannot be Gorenstein. 
In particular, as the dualising sheaf of a normal Cohen-Macaulay scheme is reflexive, $\Quot^2_{\S}(\ce)$ is not locally factorial. \\
(ii) Thompson \cite{Tho} proved that the commuting scheme $\C(\mathfrak{gl}_3)$ is normal and Cohen-Macaulay 
(implying that $\Quot^{3}_{\S}(\ce)$ is normal and Cohen-Macaulay); however, it is not known if the singularities of $\C(\mathfrak{gl}_3)$ are rational. \\
(iii) If $\Quot^{l}_{\S}(\ce)$ is Cohen-Macaulay, then the flag Quot scheme $\Quot^{l, l+1}_{\S}(\ce)$ is Cohen-Macaulay as well. Indeed, as in (\ref{eq:flag}) we can identify $\Quot^{l, l+1}_{\S}(\ce)$ with the projectivisation $\bp(\cs)$. Taking a locally free resolution 
\begin{equation*}
0\rightarrow\cf_{1}\rightarrow\cf_{0}\rightarrow\cs\rightarrow 0
\end{equation*}
of $\cs$, we see that $\bp(\cs)\subset\bp(\cf_{0})$ is the zero scheme of the regular section $\cf_{1}\rightarrow\cf_{0}\rightarrow\co(1)$ of the locally free sheaf $\cf_{1}^{\vee}(1)$ on $\bp(\cf_{0})$.
\end{remark}
\subsection{Resolution of singularities}

As $\Quot^l_{\S}(\ce)$ is typically singular (\Cref{sing}), 
it is a natural question to construct an explicit resolution of singularities. 
Ideally, one would want a modular resolution; we show, by an argument parallel to \Cref{lequals1}, 
that in the $l=2$ case such a resolution is given by the Hilbert scheme of points $\bp(\ce)^{[2]}$.
\begin{theorem}\label{lequals2}
(i) There exists a morphism of schemes
\begin{equation*}
\rho^{2}:\bp(\ce)^{[2]}\rightarrow\Quot^2_{\S}(\ce)
\end{equation*}
taking a length $2$ subscheme $\Z$ of $\bp(\ce)$ to $\ce=p_{\ast} \co(1)\rightarrow p_{\ast} \co_{\Z}(1)$. \\
(ii) The morphism $\rho^{2}$ is a resolution of singularities of $\Quot^2_{\S}(\ce)$. \\
(iii) If $\cf$ is a locally free sheaf on $\S$, then
\begin{equation*}
\rho^{2\ast}\cf^{[2]}=\cf(1)^{[2]}.
\end{equation*}
\end{theorem}

\begin{proof}
(i) Let $\cz\subset\bp(\ce)\times\bp(\ce)^{[2]}$ be the universal subscheme, and consider the diagram
\begin{equation}\label{eq:diag2}
\begin{tikzcd}
\bp(\ce) \arrow[d, "p", swap] & \bp(\ce)\times\bp(\ce)^{[2]} \arrow[l]  \arrow[d, "p\times 1"]   \\
\S & \arrow[l] \S\times\bp(\ce)^{[2]}.
\end{tikzcd}
\end{equation}
Taking the pushforward $(p\times 1)_{\ast}$ of the quotient $\co(1)\rightarrow\co_{\cz}(1)$, we obtain a morphism
\begin{equation}\label{eq:family2}
\ce\rightarrow(p\times 1)_{\ast}\co_{\cz}(1).
\end{equation}
Since $\cz\rightarrow\S\times\bp(\ce)^{[2]}$ is affine, by cohomology and base change, the map (\ref{eq:family2}) restricts to 
\begin{equation*}
\rho^{2}(\Z):\ce=p_{\ast} \co(1)\rightarrow p_{\ast}\co_{\Z}(1) \quad \mathrm{on \ each} \quad \S\times\{\Z\}.
\end{equation*}
The sheaf $p_{\ast}\co_{\Z}(1)$ is of length $2$, since its support is contained in $p(\Z)$ and
\begin{equation*}
h^0(\S, p_{\ast}\co_{\Z}(1))) = h^0(\bp(\ce), \co_{\Z}(1)) = 2.
\end{equation*}
Because $\co_{\cz}(1)$ is flat over $\bp(\ce)^{[2]}$, $(p\times 1)_{\ast}\co_{\cz}(1)$ is a flat family of length $2$ sheaves parametrised by $\bp(\ce)^{[2]}$, and to define $\rho^{2}$ it remains to prove that $\rho^{2}(\Z)$ is surjective.

If $\Z$ is a length $2$ subscheme concentrated at a single point $z$, then we have an exact sequence of twisted ideal sheaves
\begin{equation}\label{eq:twisted}
0\rightarrow\ci_{\Z}(1)\rightarrow\ci_{z}(1)\rightarrow\co_{z}\rightarrow 0
\end{equation}
where $\co_{z}\simeq\co_{z}(1)$ is the pushforward of the ideal sheaf of $z$ in $\Z$. The morphism of sheaves $p_{\ast}\co(1)\rightarrow p_{\ast}\co_{z}(1)$ is surjective by \Cref{lequals1} (i), and the vanishing $\R^{1}p_{\ast}\co(1)=0$ implies
\begin{equation*}
\R^{1}p_{\ast}\ci_{z}(1) = 0.
\end{equation*}
The long exact sequence of higher direct images of (\ref{eq:twisted}) gives
\begin{equation*}
0\rightarrow p_{\ast}\ci_{\Z}(1)\rightarrow p_{\ast}\ci_{z}(1)\rightarrow\co_{p(z)}\rightarrow\R^{1}p_{\ast}\ci_{\Z}(1)\rightarrow 0.
\end{equation*}
To show that $\rho^{2}(\Z)$ is surjective, we show $\R^{1}p_{\ast}\ci_{\Z}(1)=0$, and for this it is enough to prove that $p_{\ast}\ci_{z}(1)\rightarrow\co_{p(z)}$ is nonzero. After twisting with an invertible sheaf, we may (as in the proof of \Cref{lequals1}) assume that $\co(1)$ is very ample. Since the global sections of a very ample line bundle separate tangent vectors, the map
\begin{equation*}
\H^{0}(\ci_{z}(1))\rightarrow\H^{0}(\co_{z})
\end{equation*}
is then nonzero.

If the scheme $\Z=\{z_{1},z_{2}\}$ is reduced and not contained in a single fibre, then
\begin{equation*}
\rho^{2}(\Z) = (\rho^{1}(z_{1}), \rho^{2}(z_{2})):p_{\ast} \co(1)\rightarrow p_{\ast} \co_{\Z}(1)=p_{\ast}\co_{z_{1}}(1)\oplus p_{\ast}\co_{z_{2}}(1),
\end{equation*}
which is surjective by \Cref{lequals1} (i). Finally, if $\Z$ is contained in a single fibre $\bp(\ce(x))$, $\Z$ defines a line $\L$ in the projective space $\bp(\ce(x))$, and $\rho^{2}(\Z)$ can be identified with the composite
\begin{equation*}
\ce\rightarrow\ce(x)=\H^0(\co_{\bp(\ce(x))}(1))\rightarrow\H^0(\co_{\L}(1))\rightarrow\H^0(\co_{\Z}(1))
\end{equation*}
of surjective restriction maps.
	
(ii) We show that $\bp(\cq_{q})^{[2]}$ is the fibre of $\rho^{2}$ over a point $q$. 
As the smooth locus of $\Quot^{2}_{\S}(\ce)$ consists of all points $q$ with $\cq_{q}=\co_{\Z}$ for some length 2 subscheme $\Z\subset\S$, 
this in particular proves (by Zariski's main theorem) that $\rho^{2}$ is an isomorphism over the smooth locus. 
Consider a length $2$ subscheme  $\Z\subset\bp(\cq_{q})$. 
The factorisation $\Z\subset\bp(\cq_{q})\subset\bp(\ce)$ induces a factorisation
\begin{equation*}
\begin{tikzcd}
p_{\ast} \co(1) \arrow[r, "\rho^{2}(\Z)"] \arrow[d] & p_{\ast}\co_{\Z}(1)  \\
p_{\ast}\co_{\bp(\cq_{q})}(1), \arrow[ru] & 
\end{tikzcd}
\end{equation*}
and we now show that the morphism
\begin{equation}\label{eq:isom}
\cq_{q}=p_{\ast}\co_{\bp(\cq_{q})}(1)\rightarrow p_{\ast}\co_{\Z}(1)
\end{equation}
is an isomorphism; the above diagram then implies that $\Z$ lies in the fibre of $\rho^2$ over $q$. As (\ref{eq:isom}) is a morphism of length $2$ sheaves, it suffices to show that it is surjective. If $\cq_q$ is of the form $\co_{\Z}$ then this is clear (the map $\Z\rightarrow\S$ is finite); if however $\cq_q\simeq\co_{s}^{\oplus 2}$ for a point $s\in\S$, then $\bp(\cq_{q})\simeq\bp_{1}$ and the surjectivity follows from the vanishing $\H^1(\co_{\bp_{1}}(-1))=0$. On the other hand, assume that a length 2 subscheme $\Z\subset\bp(\ce)$ lies in the fibre of $\rho^2$ over $q$. As $\co_{\Z}(1)$ is very ample with respect to $\Z\rightarrow\S$, the equality
\begin{equation*}
\bp(\cq_{q})=\bp(p_{\ast}\co_{\Z}(1))
\end{equation*}
of subschemes of $\bp(\ce)$ shows that $\Z\subset\bp(\cq_{q})$.
	
(iii) By \Cref{basechange}, $\rho^{2\ast}\cf^{[2]}$ is given by the pushforward of $\cf\otimes(p\times 1)_{\ast}\co_{\cz}(1)$ to $\bp(\ce)^{[2]}$, which is by the projection formula for $p\times 1$ and (\ref{eq:diag2}) equal to $\cf(1)^{[2]}$.
\end{proof}

\section{Intersection theory}\label{sec:inter}

\subsection{Relations}

By \Cref{I-III} (i), a quotient of sheaves $\ce\rightarrow\ce^{''}$ induces a closed immersion
\begin{equation*}
	\iota:\Quot^{l}_{\S}(\ce^{''})\rightarrow\Quot^{l}_{\S}(\ce).
\end{equation*}
We now show that if $\ce^{''}$ is locally free, $\iota$ induces relations in intersection theory. 

\begin{theorem}\label{relations}
Let $\ce\rightarrow\ce^{''}$ be a locally free quotient with kernel $\ce^{'}$, 
and consider the section $s\in\H^{0}(\S, \ce\otimes\ce^{'\vee})$ given by the inclusion $\ce'\rightarrow\ce$. 
The induced section $s^{[l]}\in\H^0(\Quot^l_{\S}(\ce), \ce^{'\vee[l]})$ is regular, 
and its zero scheme is $\Quot^l_{\S}(\ce^{''})\subset\Quot^l_{\S}(\ce)$. 
In particular,
\begin{equation*}
\iota_\ast[\Quot^l_{\S}(\ce^{''})] = e(\ce^{'\vee[l]}) \cap [\Quot^l_{\S}(\ce)]
\end{equation*}
holds in $\A_{\ast}(\Quot^l_{\S}(\ce))$. Moreover, we have the vanishing
\begin{equation*}
e(\ce^{\vee[l]}) = 0 \quad in \quad \A_{\ast}(\Quot^l_{\S}(\ce)).
\end{equation*}
\end{theorem}

\begin{proof}
By applying \Cref{sectiondiag} to $\cf=\ce^{'\vee}$, we obtain a commutative diagram
\begin{equation*}
\begin{tikzcd}
\Hom_{\S}(\ce^{'}, \ce)  \arrow[r, "^{[l]}"] \arrow[d] & \H^{0}(\Quot^{l}_{\S}(\ce), \cf^{[l]}) \arrow[d, "\iota^{\ast}"]  \\
\Hom_{\S}(\ce^{'}, \ce^{''})  \arrow[r, swap, "^{[l]}"]  &  \H^{0}(\Quot^{l}_{\S}(\ce^{''}), \cf^{[l]}).
\end{tikzcd}
\end{equation*}
Hence $\iota^{\ast}s^{[l]}=0^{[l]}=0$, and $\iota$ induces a morphism
\begin{equation}\label{eq:zeroid}
\Quot^{l}_{\S}(\ce^{''})\rightarrow\Z(s^{[l]})
\end{equation}
of schemes over $\Quot^{l}_{\S}(\ce)$. On the other hand, by definition of $s^{[l]}$ the restriction of 
\begin{equation*}
\ce^{'}\rightarrow\ce\rightarrow\cq
\end{equation*}
to $\S\times\Z(s^{[l]})$ vanishes. It therefore induces a quotient of $\ce^{''}$ on $\S\times\Z(s^{[l]})$, and in particular a morphism 
\begin{equation*}
\Z(s^{[l]})\rightarrow\Quot^{l}_{\S}(\ce^{''})
\end{equation*}
of schemes over $\Quot^{l}_{\S}(\ce)$, which is an inverse to (\ref{eq:zeroid}).
	
The regularity of $s^{[l]}$ can be checked in a neighbourhood of a point $q$ of $\Z(s^{[l]})$. By the same argument as in the proof of \Cref{locfree}, we may find an affine open neighbourhood $\U\subset\S$ of the support of $\cq_{q}$. Then $\Quot^{l}_{\U}(\ce\vert_{\U})$ is an open neighbourhood of $q$ in $\Quot^{l}_{\S}(\ce)$, and the restriction of
\begin{equation*}
0\rightarrow\ce^{'}\rightarrow\ce\rightarrow\ce^{''}\rightarrow 0
\end{equation*}
to $\U$ splits. Without loss of generality, we may therefore assume that
\begin{equation*}
\ce = \ce^{'}\oplus\ce^{''}
\end{equation*}
on $\S$. Consider the open subscheme $\V\subset\Quot^{l}_{\S}(\ce)$ with the following functor of points (which is naturally an open subfunctor of $h_{\Quot^{l}_{\S}(\ce)}$): a morphism $f:\T\rightarrow\V$ corresponds to a quotient $\ce\rightarrow\cq_{f}$ on $\S\times\T$ such that $\ce^{''}\rightarrow\ce\rightarrow\cq_{f}$ is surjective. The scheme $\V$ is an open neighbourhood of $\Z(s^{[l]})=\Quot^{l}_{\S}(\ce^{''})$, and by restricting the morphism $\ce^{''}\rightarrow\ce\rightarrow\cq$ on $\S\times\Quot^{l}_{\S}(\ce)$ to $\S\times\V$, we obtain a morphism of schemes
\begin{equation*}
\phi:\V\rightarrow\Quot^{l}_{\S}(\ce^{''}).
\end{equation*}
The functor of points $h_{\V/\Quot^{l}_{\S}(\ce^{''})}$ of $\V$ as a scheme over $\Quot^{l}_{\S}(\ce^{''})$ can be described as follows. Let $\T$ be a scheme over $\Quot^{l}_{\S}(\ce^{''})$ with structural morphism $g$ and corresponding quotient $\ce^{''}\rightarrow\cq_{g}$ on $\S\times\T$. If $f:\T\rightarrow\V$ is a morphism of schemes, to require that $f$ be a morphism of schemes over $\Quot^{l}_{\S}(\ce^{''})$ amounts to requiring that $\ce^{''}\rightarrow\ce\rightarrow\cq_{f}$ and $\ce^{''}\rightarrow\cq_{g}$ are the same quotient of $\ce^{''}$, while the morphism of sheaves $\ce^{'}\rightarrow\ce\rightarrow\cq_{g}$ can be arbitrary. It thus follows that the map
\begin{equation}\label{eq:relfunctor}
h_{\V / \Quot^{l}_{\S}(\ce^{''})}(\T) \rightarrow \Hom_{\S\times\T}(\ce^{'}, \cq_{g}).
\end{equation}
taking $f$ to $\ce^{'}\rightarrow\cq_{g}$ is bijective. The right hand side defines a functor on the category of schemes over $\Quot^{l}_{\S}(\ce^{''})$ (given on morphisms by pullback), and (\ref{eq:relfunctor}) can be viewed as a functorial isomorphism. Consider now the vector bundle $\V(\ce^{'\vee[l]})$ defined by the locally free sheaf $\ce^{'\vee[l]}$ on $\Quot^{l}_{\S}(\ce^{''})$, and recall that 
\begin{equation*}
h_{\V(\ce^{'\vee[l]}) / \Quot^{l}_{\S}(\ce^{''})}(\T) = \H^{0}(\T, g^{\ast}\ce^{'\vee[l]}).
\end{equation*}
We now describe a functorial isomorphism
\begin{equation*}
\Phi: h_{\V(\ce^{'\vee[l]}) / \Quot^{l}_{\S}(\ce^{''})} \rightarrow h_{\V / \Quot^{l}_{\S}(\ce^{''})}.
\end{equation*}
In view of (\ref{eq:relfunctor}), it suffices to observe that \Cref{basechange} implies the existence of a canonical isomorphism
\begin{equation*}
\Phi(\T): \H^{0}(\T, g^{\ast}\ce^{'\vee[l]})\rightarrow\Hom_{\S\times\T}(\ce^{'}, \cq_{g})
\end{equation*}
for every scheme $\T$ over $\Quot^{l}_{\S}(\ce^{''})$. By \Cref{basechange}, we have
\begin{equation*}
\phi^{\ast}\ce^{'\vee[l]}=\ce^{'\vee[l]}\vert_{\V},
\end{equation*}
and it is clear from (\ref{eq:relfunctor}) and the definition of $\Phi$ that $s^{[l]}\vert_{\V}$ is the universal section associated to $\V(\ce^{'\vee[l]})$. This in particular establishes the regularity of $s^{[l]}$. As $s^{[l]}$ is regular, the embedding $\iota$ is regular, and a standard construction in intersection theory \cite{Ful} yields a Gysin map
\begin{equation*}
\iota^{\ast}:\A_{\ast}(\Quot^l_{\S}(\ce))\rightarrow\A_{\ast}(\Quot^l_{\S}(\ce^{''}))
\end{equation*}
satisfying $\iota^{\ast}[\Quot^l_{\S}(\ce)]=[\Quot^l_{\S}(\ce^{''})]$ and
\begin{equation}\label{eq:gysin}
\iota_{\ast}\circ\iota^{\ast}=e(\ce^{'\vee[l]})\cap(-).
\end{equation}
Finally, we consider the trivial quotient $\ce\rightarrow\ce^{''}=0$, so that $s$ is the identity map of $\ce$. At any given point $q$ of $\Quot^{l}_{\S}(\ce)$, the value $s^{[l]}(q)\in\Hom(\ce, \cq_{q})$ is the quotient $\ce\rightarrow\cq_{q}$ corresponding to $q$. In particular, $s^{[l]}$ is a nowhere vanishing section of $\ce^{\vee[l]}$.
\end{proof}
We should observe that since the codimension of $\Quot^{l}_{\S}(\ce^{''})$ in $\Quot^{l}_{\S}(\ce)$ is equal to the rank of $\ce^{\vee'[l]}$,  the regularity of $s^{[l]}$ is also a consequence of \Cref{singularities}. Our proof shows that
the identification $\Z(s^{[l]})=\Quot^{l}_{\S}(\ce^{''})$ and the regularity of $s^{[l]}$ hold for any smooth projective scheme $\S$ of dimension $d$; however, the equality
\begin{equation*}
\iota^{\ast}[\Quot^l_{\S}(\ce)]=[\Quot^l_{\S}(\ce^{''})]
\end{equation*}
requires $\Quot^l_{\S}(\ce)$ to be of pure dimension. If $d\geqslant 3$, this fails for $l\gg 0$, but holds for small $l$ \cite{JoS}. A direct consequence of the regularity of $s^{[l]}$ is the following result (parallel to \Cref{smoothI}) for the Cohen-Macaulay and Gorenstein loci, $\Quot^{l}_{\S}(\ce)^{\mathrm{CM}}$ and $\Quot^{l}_{\S}(\ce)^{\mathrm{G}}$, respectively.
\begin{corollary}
For any locally free quotient $\ce\rightarrow\ce^{''}$, we have
\begin{align*}
\Quot^{l}_{\S}(\ce^{''})^{\mathrm{CM}} &=\Quot^{l}_{\S}(\ce^{''})\cap\Quot^{l}_{\S}(\ce)^{\mathrm{CM}}, \\
\Quot^{l}_{\S}(\ce^{''})^{\mathrm{G}} &=\Quot^{l}_{\S}(\ce^{''})\cap\Quot^{l}_{\S}(\ce)^{\mathrm{G}}.
\end{align*}
\end{corollary}

\begin{remark}
Taking $l=1$ and identifying $\Quot^{1}_{\S}(\ce) = \bp(\ce)$ as in \Cref{lequals1}, we have $\ce^{\vee[1]} = \ce^{\vee}(1)$ and the relation
$e(\ce^{\vee[1]})=0$ reduces to the relation
\begin{equation*}
\sum_{k=0}^{r}(-1)^{r-k}c_{1}(\co(1))^{k}c_{r-k}(\ce)=0,
\end{equation*}
which was first considered by Grothendieck \cite{Gro2}.
\end{remark}

\begin{remark}
Consider a commutative diagram of locally free quotients
\begin{equation*}
\begin{tikzcd}
\ce \arrow[r] \arrow[d] & \ce^{''} \arrow[d]   \\
\cf \arrow[r] & \cf^{''},
\end{tikzcd}
\end{equation*}
and write $\cd$ and $\cd^{''}$ for the kernels of $\ce\rightarrow\cf$ and $\ce^{''}\rightarrow\cf^{''}$, respectively. By applying \Cref{relations} to the induced commutative diagram of embeddings of Quot schemes
\begin{equation*}
		\begin{tikzcd}
			\Quot^{l}_{\S}(\cf^{''}) \arrow[r] \arrow[d] & \Quot^{l}_{\S}(\cf) \arrow[d]   \\
			\Quot^{l}_{\S}(\ce^{''}) \arrow[r] & \Quot^{l}_{\S}(\ce),
		\end{tikzcd}
\end{equation*}
we obtain the commutation relation
\begin{equation*}
e(\cd^{\vee[l]})\cap e(\cf^{'\vee[l]})\cap[\Quot^{l}_{\S}(\ce)] = e(\ce^{'\vee[l]})\cap e(\cd^{''\vee[l]})\cap[\Quot^{l}_{\S}(\ce)].
\end{equation*}
\end{remark}

We also have the following virtual analogue of \Cref{relations}.

\begin{theorem}\label{virclass}
For any locally free quotient $\ce\rightarrow\ce^{''}$ with kernel $\ce^{'}$, we have
\begin{equation*}
\iota_\ast[\Quot^{l}_{\S}(\ce^{''})]^{\vir}=e(\ce^{'\vee[l]})\cap[\Quot^{l}_{\S}(\ce)]^{\vir}.
\end{equation*}
\end{theorem}

\begin{proof}
By the property (\ref{eq:gysin}) of the Gysin map $\iota^{\ast}$ , it suffices to show that
\begin{equation*}
\iota^{\ast}[\Quot^l_{\S}(\ce)]^{\vir}=[\Quot^l_{\S}(\ce^{''})]^{\vir}.
\end{equation*}
We will use the Siebert formula
\begin{equation*}
[\Quot^l_{\S}(\ce)]^{\vir} = \left\lbrace s(\T^{\vir}_{\Quot^l_{\S}(\ce)})\cap c_{\F}(\Quot^l_{\S}(\ce))\right\rbrace_{lr}.
\end{equation*}
As in (\ref{eq:comm2}), using the vanishing of higher derived pullbacks (\ref{eq:vanishing}), $\iota$ induces a morphism of exact sequences
\begin{equation*}
	\begin{tikzcd}
	 0 \arrow[r] & (1\times\iota)^{\ast}\cs \arrow[r] \arrow[d] &	\ce \arrow[r] \arrow[d] & (1\times\iota)^{\ast}\cq \arrow[d, "\isom"] \arrow[r] & 0 \\
	0 \arrow[r] & \cs^{''} \arrow[r] &	\ce^{''} \arrow[r] & \cq^{''} \arrow[r] & 0
	\end{tikzcd}
\end{equation*}
on $\S\times\Quot^{l}_{\S}(\ce^{''})$, where the bottom row is given by the universal quotient of $\Quot^{l}_{\S}(\ce^{''})$. Since the kernel of the middle vertical map is $\ce^{'}$, we obtain an exact triangle
\begin{equation*}
(1\times\iota)^\ast\cs\rightarrow\cs^{''}\rightarrow \ce^{'}[1]\rightarrow (1\times\iota)^{\ast}\cs[1].
\end{equation*}
Applying the functor $\R\SHom_{\pi^{''}}(-,\cq^{''})$, where $\pi^{''}:\S\times\Quot^{l}_{\S}(\ce^{''})\rightarrow \Quot^{l}_{\S}(\ce^{''})$ is the projection, yields an exact triangle of the form
\begin{equation}\label{eq:compatibility}
\T^{\vir}_{\Quot^l_{\S}(\ce^{''})}\rightarrow\L \iota^{\ast}\T^{\vir}_{\Quot^l_{\S}(\ce)}\rightarrow\ce^{'\vee[l]}\rightarrow\T^{\vir}_{\Quot^l_{\S}(\ce^{''})}[1].
\end{equation}
We now prove the equality
\begin{equation}\label{eq:fulton}
\iota^{\ast}c_{\F}(\Quot^{l}_{\S}(\ce)) = c(\ce^{'\vee[l]})\cap c_{\F}(\Quot^{l}_{\S}(\ce^{''}))
\end{equation}
Using \Cref{tautcomp} (i), it suffices to show that for a smooth scheme $\A$ containing $\Quot^{l}_{\S}(\ce)$
\begin{equation*}
\iota^{\ast}s(\C_{\Quot^{l}_{\S}(\ce)/\A}) = c(\ce^{'\vee[l]})\cap s(\C_{\Quot^{l}_{\S}(\ce^{''})/\A}).
\end{equation*}
For the latter equality, we consider the sequence of normal cones
\begin{equation}\label{eq:normalcones}
\C_{\Quot^{l}_{\S}(\ce^{''})/\Quot^{l}_{\S}(\ce)}\rightarrow\C_{\Quot^{l}_{\S}(\ce^{''})/\A}\rightarrow\C_{\Quot^{l}_{\S}(\ce)/\A}|_{\Quot^{l}_{\S}(\ce^{''})}.
\end{equation}
and show that it is exact in the sense of Example 4.1.6 of \cite{Ful}. By \Cref{relations}, the cone $\C_{\Quot^{l}_{\S}(\ce^{''})/\Quot^{l}_{\S}(\ce)}$ is the vector bundle $\V(\ce^{'\vee[l]})$. Since the exactness of (\ref{eq:normalcones}) can be verified étale locally on $\Quot^{l}_{\S}(\ce^{''})$, we may assume, as in the proof of \Cref{relations}, that $\ce=\ce^{'}\oplus\ce^{''}$ and check the exactness on the open subscheme $\V$. Using the identification of $\V$ with the vector bundle associated to $\ce^{'\vee[l]}$, one can, after passing to a open cover on which $\ce^{'\vee[l]}$ is trivial, reduce to proving the exactness of the sequence of normal cones associated to $\U\subset\U\times\ba^{n}\subset\A=\A^{'}\times\ba^{n}$, where $\A^{'}$ is a smooth scheme containing $\U$, and $n$ is the rank of $\ce^{'\vee[l]}$. By the proof of Theorem 6.5 of \cite{Ful}, in this situation one has a canonical isomorphism
\begin{equation*}
\C_{\U/\A}\xrightarrow{\sim}\ba^{n}_{\U}\times_{\U}\C_{\U/\A^{'}},
\end{equation*}
as well as
\begin{equation*}
\C_{\U/\U\times\ba^{n}}\xrightarrow{\sim}\ba^{n}_{\U}, \quad \mathrm{and} \quad \C_{\U\times\ba^{n}/\A}\vert_{\U}\xrightarrow{\sim}\C_{\U/\A^{'}},
\end{equation*}
which establish the exactness in this case. Combining (\ref{eq:compatibility}) and (\ref{eq:fulton}) with Siebert's formula,
\begin{align*}
\iota^{\ast}[\Quot^l_{\S}(\ce)]^{\vir} &=\left\{\iota^{\ast}s(\T^{\vir}_{\Quot^l_{\S}(\ce)})\cap \iota^{\ast}c_{\F}(\Quot^l_{\S}(\ce))\right\}_{r''l} \\
&=\left\{s(\T^{\vir}_{\Quot^l_{\S}(\ce^{''})})\cap s(\ce^{'\vee[l]})\cap c(\ce^{'\vee[l]})\cap c_{\F}(\Quot^l_{\S}(\ce^{''}))\right\}_{r''l} \\
&=\left\{s(\T^{\vir}_{\Quot^l_{\S}(\ce^{''})})\cap c_{\F}(\Quot^l_{\S}(\ce^{''}))\right\}_{r''l}=[\Quot^l_{\S}(\ce^{''})]^{\vir},
\end{align*}
where we have written $r^{''}$ for the rank of $\ce^{''}$.
\end{proof}
\subsection{Integrals}

A key problem in the intersection theory of $\Quot^l_{\S}(\ce)$ is to compute integrals of the form
\begin{equation}\label{eq:taut}
\int_{\Quot^l_{\S}(\ce)} P(\cf_{1}^{[l]}, \ldots, \cf_{m}^{[l]}),
\end{equation}
where $P(\cf_{1}^{[l]}, \ldots, \cf_{m}^{[l]})$ denotes any polynomial in the Chern classes of the tautological sheaves $\cf_{1}^{[l]}, \ldots, \cf_{m}^{[l]}$. These integrals are compatible with properties (i) and (ii) of \Cref{I-III} in the following sense:
\begin{proposition}\label{integrals}
Let $\cf_{1}, \ldots, \cf_{m}$ be locally free sheaves on $\S$. \\
(i) For any locally free quotient $\ce\rightarrow\ce^{''}$ with kernel $\ce^{'}$, we have	
\begin{equation*}
\int_{\Quot^l_{\S}(\ce)} e(\ce^{'\vee[l]})P(\cf_{1}^{[l]}, \ldots, \cf_{m}^{[l]}) =\int_{\Quot^l_{\S}(\ce^{''})} P(\cf_{1}^{[l]}, \ldots, \cf_{m}^{[l]}).
\end{equation*}
(ii) For any invertible sheaf $\cl$ on $\S$, we have
\begin{equation*}
\int_{\Quot^l_{\S}(\ce\otimes\cl)} P(\cf_{1}^{[l]}, \ldots, \cf_{m}^{[l]}) = \int_{\Quot^l_{\S}(\ce)} P((\cf_{1}\otimes\cl)^{[l]}, \ldots, (\cf_{m}\otimes\cl)^{[l]}).
\end{equation*}
\end{proposition}

\begin{proof}
Part (i) follows from \Cref{relations} and \Cref{tautcomp} (i), while part (ii) is a consequence of \Cref{I-III} (ii) and \Cref{tautcomp} (ii).
\end{proof}
This result allows one to compute certain integrals over $\Quot^l_{\S}(\ce)$. Consider $\ce$ on $\S$ such that there exists a quotient $\ce\rightarrow\cl$, where $\cl$ is invertible. Denoting by $\ce^{'}$ the kernel of $\ce\rightarrow\cl$, we have
\begin{align*}
\int_{\Quot^l_{\S}(\ce)} e(\ce^{'\vee[l]})P(\cf_{1}^{[l]}, \ldots, \cf_{m}^{[l]}) &= \int_{\Quot^l_{\S}(\cl)} P(\cf_{1}^{[l]}, \ldots, \cf_{m}^{[l]}) \\
&= \int_{\S^{[l]}} P((\cf_{1}\otimes\cl)^{[l]}, \ldots, (\cf_{m}\otimes\cl)^{[l]})
\end{align*}
by (i) and (ii). Such integrals can therefore be computed in terms of integrals over $\S^{[l]}$.
\begin{example}
If $\cm$ is an invertible sheaf on a K3 surface $\S$, then
\begin{equation*}
\int_{\Quot^l_{\S}(\ce)} e(\ce^{'\vee[l]})c_{1}(\cm^{[l]})^{2l} = \frac{(2l)!}{l!2^{l}}\left(c_{1}(\cl\otimes\cm)^{2} + 2(1-l) \right)^{l}
\end{equation*}
using the results of Beauville \cite{Bea}.
\end{example}

It should be observed that to compute an integral of the form (\ref{eq:taut}), one can by (ii) assume that $\ce$ is globally generated and then use (i) to reduce to the case $\ce=\co^{\oplus r}$. To compute the integrals (\ref{eq:taut}) in this case, it would be very important to prove their compatibility with property (iii) of \Cref{I-III}. This would come down to proving a localization theorem for the singular scheme $\Quot^{l}_{\S}(\co^{\oplus r})$, allowing one to express integrals over $\Quot^{l}_{\S}(\co^{\oplus r})$ as sums of integrals over products of Hilbert schemes of points (using the description of the fixed locus given in \cite{Bif}). In the same vein, it would be crucial to obtain a structure (universality) result for the integrals (\ref{eq:taut}), generalizing the result of \cite{EGL} in the $r=1$ case, expressing (\ref{eq:taut}) as a universal polynomial in the intersection numbers given by combinations of the Chern classes of $\ce\otimes\cf_{1}, \ldots, \ce\otimes\cf_{m}$ and the Chern classes of $\S$. In the context of integrals over the virtual class $[\Quot^{l}_{\S}(\ce)]^{\vir}$, these questions are much easier \cite{GrP, OP, Sta}.

\begin{remark}
As in the $r=1$ case, one should more generally consider integrals of the form (\ref{eq:taut}) with insertions given by MacPherson's Chern class \cite{Mac}.
\end{remark}

A notable special case of (\ref{eq:taut}) are the integrals
\begin{equation}\label{eq:Segre}
\int_{\Quot^l_{\S}(\ce)} s_{l(r+1)}(\cl^{[l]}),
\end{equation}
where $\cl$ is an invertible sheaf on $\S$. In the rank $r=1$ case, these were computed by Marian, Oprea, and Pandharipande \cite{MOP1, MOP2} (using a result of Voisin \cite{Voi}), resolving a long-standing conjecture of Lehn \cite{Leh}. The problem of extending this formula to $\Quot^l_{\S}(\ce)$ was mentioned in the recent paper of Oprea and Pandharipande \cite{OP}. For $l=1$ one readily finds, using \Cref{lequals1}, that (\ref{eq:Segre}) is equal to $(-1)^{r+1} s_{2}(\ce\otimes\cl)$. In the next section, we use \Cref{lequals2} to evaluate (\ref{eq:Segre}) in the first nontrivial (singular) case $l=2$.
\subsection{Higher rank Severi formula}

The following result is well-known; we prove it primarily because of the lack of a suitable reference.
\begin{proposition}\label{Rank1}
Let $\cl$ be an invertible sheaf on a smooth projective scheme $\X$. Then
\begin{equation*}
2\int_{\X^{[2]}}s_{2d}(\cl^{[2]}) = \left( \int_{\X} c_{1}(\cl)^{d} \right)^{2} -\sum_{k=0}^{d}\binom{2d+1}{d-k}\int_{\X} c_{1}(\cl)^{d-k}s_{k}(\X),
\end{equation*}
where $d$ is the dimension of $\X$, and $s(\X)$ the Segre class of its tangent sheaf.
\end{proposition}
\begin{proof}
Consider the diagram
\begin{equation}
\begin{tikzcd}
& \X^{[1, 2]} \arrow[rd, "\phi"] \arrow[ld, swap, "\psi"] &  \\
\X\times\X & & \X^{[2]}
\end{tikzcd}
\end{equation}
given by the flag Hilbert scheme. Here $\psi$ can be identified with the blow-up of $\X\times\X$ along the diagonal, and $\phi$ is a double covering. On $\X^{[1, 2]}$ we have an exact sequence of the form
\begin{equation}\label{eq:ses}
0\rightarrow\phi^{\ast}\cl^{[2]}\rightarrow\psi^{\ast}(\cl\boxplus\cl)\rightarrow\iota_{\ast}\cl\rightarrow 0,
\end{equation}
where $\iota$ is the inclusion of the exceptional divisor $\bp(\Omega^{1}_{\X})$ of $\psi$. We have
\begin{equation*}
c(\iota_{\ast} \cl)=1+\iota_{\ast} s(\cl(1))
\end{equation*}
by the Riemann-Roch theorem without denominators, and hence
\begin{equation*}
2\int_{\X^{[2]}} s(\cl^{[2]}) = \int_{\X^{[1, 2]}} s\left( \psi^{\ast}(\cl\boxplus\cl)\right)c(\iota_{\ast} \cl)
= \int_{\X\times\X} s\left(\cl\boxplus\cl\right)+\int_{\bp(\Omega^{1}_{\X})} s(\cl)^{2} s(\cl(1))
\end{equation*}
by (\ref{eq:ses}). The first integral is given by
\begin{equation*}
\int_{\X\times\X} s\left(\cl\boxplus\cl\right) = \left( \int_{\X} c_{1}(\cl)^{d}\right)^{2}.
\end{equation*}
As the degree $2d-1$ term of $s(\cl)^{2} s(\cl(1))$ is given by
\begin{equation*}
\left\{ s(\cl)^{2} s(\cl(1)) \right\}_{2d-1}
=-\sum_{k=0}^{2d-1} \binom{2d+1}{k+2} c_{1}(\cl)^{2d-1-k} c_{1}(\co(1))^k.
\end{equation*}
and since the pushforward map $\A_{\ast}(\bp(\Omega^{1}_{\X}))\rightarrow\A_{\ast}(\X)$ 
takes $c_{1}(\co(1))^k$ to $s_{k-d+1}(\X)$, we find
\begin{equation*}
\int_{\bp(\Omega^{1}_{\X})} s(\cl)^{2} s(\cl(1)) = -\sum_{k=0}^{d}\binom{2d+1}{d-k}\int_{\X} c_{1}(\cl)^{d-k}s_{k}(\X). \qedhere
\end{equation*}
\end{proof}
If we take $\X\subset\bp_{2d+1}$ and $\cl$ coming from restricting $\co(1)$, 
then the integral of $s_{2d}(\cl^{[2]})$ can be interpreted as the number of secants of $\X$ through a general point $p$ of $\bp_{2d+1}$, 
or equivalently the number of improper double points of the projection $\X\rightarrow\bp_{2d}$ associated to $p$. 
In this context, \Cref{Rank1} is due to Severi \cite{Sev}. When $\X=\S$ is a surface, \Cref{Rank1} reduces to 
\begin{equation*}
2\int_{\S^{[2]}} s_{4}(\cl^{[2]}) = s_{2}(\cl)^{2} - 10s_{2}(\cl) + 5s_{1}(\cl)s_{1}(\S) - s_{2}(\S).
\end{equation*}
We now extend the latter formula to $\Quot^{2}_{\S}(\ce)$ for a locally free sheaf $\ce$ of arbitrary rank $r$. Reflecting \Cref{integrals} (ii), the right-hand side of the higher rank formula depends on the rank $r$, the Segre classes of $\ce\otimes\cl$, as well as an additional term coming from the discriminant
\begin{equation*}
\Delta(\ce)=c_{2}(\SEnd(\ce)).
\end{equation*}

\begin{theorem}\label{higherseveri}
For any invertible sheaf $\cl$ on $\S$
\begin{align*}
2\int_{\Quot^{2}_{\S}(\ce)} s_{2r+2}(\cl^{[2]}) &= -\frac{r+2}{2}\Delta(\ce) +  s_{2}(\ce\otimes\cl)^{2}-(2r+3)(r+1)s_{2}(\ce\otimes\cl) \\
&+ \frac{1}{6}(2r+3)(r+2)(r+1)s_{1}(\ce\otimes\cl)s_{1}(\S)-\binom{r+3}{4}s_{2}(\S).
\end{align*}
\end{theorem}

\begin{proof}
\Cref{lequals2} (iii) implies the equality
\begin{equation*}
\int_{\Quot^{2}_{\S}(\ce)} s_{2r+2}(\cl^{[2]})=\int_{\bp(\ce)^{[2]}} s_{2r+2}(\cl(1)^{[2]})
\end{equation*}
and by \Cref{Rank1} we have to compute the integrals
\begin{equation}\label{eq:integrals}
\int_{\bp(\ce)} c_{1}(\cl(1))^{r+1-k} s_{k}(\bp(\ce)).
\end{equation}
Writing $\zeta=c_{1}(\co(1))$, the pushforward map $p_{\ast}:\A_{\ast}(\bp(\ce))\rightarrow\A_{\ast}(\S)$ satisfies
\begin{equation}
p_{\ast}\zeta^{r-1} = 1, \quad 	p_{\ast}\zeta^{r} = -s_{1}(\ce), \quad p_{\ast}\zeta^{r+1} = s_{2}(\ce),
\end{equation}
and $p_{\ast}\zeta^{k}=0$ for $k\leqslant r-2$. To determine (\ref{eq:integrals}) it therefore suffices to compute the coefficients of $\zeta^{r+1}, \zeta^{r}, \zeta^{r-1}$ in $c_{1}(\cl(1))^{r+1-k} s_{k}(\bp(\ce))$. We have
\begin{equation*}
c_{1}(\cl(1))^{r+1-k} = \zeta^{r+1-k} + (r+1-k)c_{1}(\cl)\zeta^{r-k} + \binom{r+1-k}{2}c_{1}(\cl)^{2}\zeta^{r-1-k} + \cdots.
\end{equation*}
Using the exact sequence of differentials associated to $p$ and the Euler sequence, we find
\begin{equation*}
s(\bp(\ce)) = s(\S)s(\Theta_{p}) = s(\S)s(\ce^{\vee}(1))
\end{equation*}
and in particular
\begin{gather*}
(-1)^{k}s_{k}(\bp(\ce)) = \binom{r-1+k}{k}\zeta^{k} + \left\lbrace \binom{r-1+k}{k-1}s_{1}(\ce) - \binom{r-2+k}{k-1}s_{1}(\S)\right\rbrace \zeta^{k-1} \\
+ \left\lbrace \binom{r-1+k}{k-2}s_{2}(\ce) - \binom{r-2+k}{k-2}s_{1}(\S)s_{1}(\ce) + \binom{r-3+k}{k-2}s_{2}(\S)\right\rbrace \zeta^{k-2} + \cdots.
\end{gather*}
Hence $(-1)^{k}c_{1}(\cl(1))^{r+1-k} s_{k}(\bp(\ce))$ is given by
\begin{gather*}
\left\lbrace \binom{r-1+k}{k} + \binom{r-1+k}{k-2}\right\rbrace s_{2}(\ce) -\binom{r-1+k}{k-1}s_{1}(\ce)^{2} + \binom{r-3+k}{k-2}s_{2}(\S) \\
+(r+1-k) \left\lbrace \binom{r-1+k}{k-1} - \binom{r-1+k}{k} \right\rbrace c_{1}(\cl)s_{1}(\ce) +\binom{r+1-k}{2}\binom{r-1+k}{k}c_{1}(\cl)^{2} \\
+ \left\lbrace \binom{r-2+k}{k-1}s_{1}(\ce) - \binom{r-2+k}{k-2}s_{1}(\ce) - (r+1-k)\binom{r-2+k}{k-1}c_{1}(\cl) \right\rbrace s_{1}(\S).
\end{gather*}
Using the identities
\begin{align*}
s_{2}(\ce\otimes\cl) &= s_{2}(\ce)-(r+1)s_{1}(\ce)c_{1}(\cl)+\binom{r+1}{2}c_{1}(\cl)^{2}, \\
\Delta(\ce) &=-2rs_{2}(\ce)+(r+1)s_{1}(\ce)^{2},
\end{align*} 
we find that the integrals (\ref{eq:integrals}) are given by
\begin{equation*}
\int_{\bp(\ce)} c_{1}(\cl(1))^{r+1-k} s_{k}(\bp(\ce)) = \alpha_{k}\Delta(\ce) + \beta_{k} s_{2}(\ce\otimes\cl) + \gamma_{k}s_{1}(\ce\otimes\cl)s_{1}(\S) + \delta_{k}s_{2}(\S),
\end{equation*}
where we have written
\begin{gather*}
\alpha_{k}  = (-1)^{k+1}\frac{(r-1+k)!}{(k-1)!(r+1)!}, \quad \beta_{k}= (-1)^{k} (r+1-k)(r-k)\frac{(r-1+k)!}{k!(r+1)!}, \\
\gamma_{k} = (-1)^{k} (r+1-k)\frac{(r-2+k)!}{r!(k-1)!}, \quad \delta_{k}= (-1)^{k} \binom{r-3+k}{k-2}.
\end{gather*}
It remains to evaluate the binomial sums arising from \Cref{Rank1}. We find
\begin{align*}
\sum_{k=0}^{r+1} \binom{2r+3}{r+1-k}\alpha_{k} &= \frac{r+2}{2}, \\
\sum_{k=0}^{r+1} \binom{2r+3}{r+1-k} \beta_{k} &= (2r+3)(r+1), \\
\sum_{k=0}^{r+1} \binom{2r+3}{r+1-k}\gamma_{k} &= -\frac{1}{6}(2r+3)(r+2)(r+1), \\
\sum_{k=0}^{r+1} \binom{2r+3}{r+1-k}\delta_{k} &= \binom{r+3}{4}.
\end{align*}
One can prove these identities by identifying the sums as hypergeometric sums, 
or, more elementarily, as convolution sums arising from the expansion of $(1+z)^{a}(1+z)^{b} = (1+z)^{a+b}$ for some $a, b$.
\end{proof}

\section{Cohomology}\label{sec:cohom}

\subsection{Structure sheaf}

\Cref{singularities} allows us to compute the cohomology of the structure sheaf of the Quot scheme $\Quot^{l}_{\S}(\ce)$; the result is, perhaps surprisingly, independent of $\ce$.
\begin{theorem}\label{coh}
If the singularities of $\Quot^{l}_{\S}(\ce)$ are rational, then
\begin{equation*}
\H^{\ast}(\Quot^{l}_{\S}(\ce), \co_{\Quot^{l}_{\S}(\ce)})\xrightarrow{\sim}\S^{l}\H^{\ast}(\S, \co_{\S}).
\end{equation*}
\end{theorem}

\begin{proof}
By \Cref{basic} (iii), $\Quot^{l}_{\S}(\ce)$ is birational to $\bp(\ce)$; let $\Gamma$ be the closure of the graph of the corresponding birational map. 
By taking a resolution of singularities $\tilde{\Gamma}\rightarrow\Gamma$, we obtain a roof diagram
\begin{equation*}
	\begin{tikzcd}
		& \tilde{\Gamma} \arrow[rd] \arrow[ld] &  \\
		\Quot^{l}_{\S}(\ce) & & \bp(\ce)^{(l)}
	\end{tikzcd}
\end{equation*}
of proper birational morphisms. Since $\bp(\ce)^{(l)}$ has rational singularities, $\tilde{\Gamma}\rightarrow\bp(\ce)^{(l)}$ induces an isomorphism on cohomology
\begin{equation*}
\H^{\ast}(\tilde{\Gamma}, \co_{\tilde{\Gamma}})\xrightarrow{\sim} \H^{\ast}(\bp(\ce)^{(l)}, \co_{\bp(\ce)^{(l)}}).
\end{equation*}
As we assume that $\Quot^{l}_{\S}(\ce)$ also has rational singularities, the morphism $\tilde{\Gamma}\rightarrow\Quot^{l}_{\S}(\ce)$ induces
\begin{equation*}
\H^{\ast}(\tilde{\Gamma}, \co_{\tilde{\Gamma}})\xrightarrow{\sim} \H^{\ast}(\Quot^{l}_{\S}(\ce), \co_{\Quot^{l}_{\S}(\ce)}).
\end{equation*}
Finally, using the projection $\bp(\ce)\rightarrow\S$, we have
\begin{equation*}
\H^{\ast}(\bp(\ce)^{(l)}, \co_{\bp(\ce)^{(l)}})\xrightarrow{\sim}\H^{\ast}(\bp(\ce)^{l}, \co_{\bp(\ce)^{l}})^{\mathfrak{S}_{l}}
\xrightarrow{\sim}\S^{l}\H^{\ast}(\bp(\ce), \co_{\bp(\ce)})\xrightarrow{\sim}\S^{l}\H^{\ast}(\S, \co_{\S}). \qedhere
\end{equation*}
\end{proof}

\begin{corollary}\label{coh2}
We have $\H^{\ast}(\Quot^{2}_{\S}(\ce), \co_{\Quot^{2}_{\S}(\ce)})\xrightarrow{\sim}\S^{2}\H^{\ast}(\S, \co_{\S})$.
\end{corollary}
\begin{proof}
This follows from \Cref{ratsing} and \Cref{coh}.
\end{proof}
\subsection{Tautological sheaves}

As for the cohomology of tautological sheaves, we propose the following conjecture.
\begin{conjecture}\label{tautcoh}
For any locally free sheaf $\cf$ on $\S$,
\begin{equation*}
 \H^{\ast}(\Quot^{l}_{\S}(\ce), \cf^{[l]})\xrightarrow{\sim}\H^{\ast}(\S, \ce\otimes\cf)\otimes\S^{l-1}\H^{\ast}(\S, \co_{\S}).
\end{equation*}
\end{conjecture}
This conjecture naturally falls into two parts. On the one hand, an isomorphism
\begin{equation}\label{eq:tautcoh1}
\H^{\ast}(\Quot^{l}_{\S}(\ce), \cf^{[l]})\xrightarrow{\sim}\H^{\ast}(\S, \ce\otimes\cf)\otimes\H^{\ast}(\Quot^{l-1}_{\S}(\ce), \co_{\Quot^{l-1}_{\S}(\ce)}),
\end{equation}
which naturally suggests a flag Quot scheme (\ref{eq:flag}) approach, and on the other hand an isomorphism
\begin{equation}\label{eq:tautcoh2}
\H^{\ast}(\Quot^{l-1}_{\S}(\ce), \co_{\Quot^{l-1}_{\S}(\ce)})\xrightarrow{\sim}\S^{l-1}\H^{\ast}(\S, \co_{\S}).
\end{equation}
The latter is covered by \Cref{coh}. The $l=1$ case is an immediate consequence of \Cref{lequals1} (iii). We now show that \Cref{tautcoh} holds in the first nontrivial case $l=2$. 
\begin{theorem}\label{tautcoh2}
For any locally free sheaf $\cf$ on $\S$,
\begin{equation*}
\H^{\ast}(\Quot^{2}_{\S}(\ce), \cf^{[2]})\xrightarrow{\sim}\H^{\ast}(\S, \ce\otimes\cf)\otimes\H^{\ast}(\S, \co_{\S}).
\end{equation*}
\end{theorem}

\begin{proof}
\Cref{lequals2} implies
\begin{equation}
\H^{\ast}(\Quot^{2}_{\S}(\ce), \cf^{[2]})\xrightarrow{\sim}\H^{\ast}(\bp(\ce)^{[2]}, \cf(1)^{[2]}).
\end{equation}
To compute $\H^{\ast}(\bp(\ce)^{[2]}, \cf(1)^{[2]})$, we consider the flag Hilbert scheme
\begin{equation*}
	\begin{tikzcd}
		\bp(\ce)^{[1,2]} \arrow[r, "f_{2}"] \arrow[d, "f_{1}", swap] & \bp(\ce)\times\bp(\ce)^{[2]} \arrow[d, "g_{2}"]   \\
		\bp(\ce)\times\bp(\ce) \arrow[r, "g_{1}", swap] & \bp(\ce)
	\end{tikzcd}
\end{equation*}
where $f_{1}$ can be identified with the blow up of $\bp(\ce)\times\bp(\ce)$ along the diagonal, $f_{2}$ is an isomorphism onto the universal subscheme $\cz\subset\bp(\ce)\times\bp(\ce)^{[2]}$, and $g_{1}$ is the second projection. In particular, we have
\begin{equation*}
\H^{\ast}(\bp(\ce)^{[1,2]}, f_{1}^{\ast} g_{1}^{\ast}\cf(1)) = \H^{\ast}(\bp(\ce)^{[1,2]}, f_{2}^{\ast} g_{2}^{\ast}\cf(1)).
\end{equation*}
On the one hand, 
\begin{equation*}
\H^{\ast}(\bp(\ce)^{[1,2]}, f_{2}^{\ast} g_{2}^{\ast}\cf(1)) \xrightarrow{\sim} \H^{\ast}(\bp(\ce)\times\bp(\ce)^{[2]}, \co_{\cz}\otimes g_{2}^{\ast}\cf(1)) \xrightarrow{\sim} \H^{\ast}(\bp(\ce)^{[2]}, \cf(1)^{[2]}).
\end{equation*}
On the other hand, we have an isomorphism
\begin{equation*}
\H^{\ast}(\bp(\ce)^{[1,2]}, f_{1}^{\ast} g_{1}^{\ast}\cf(1)) \xrightarrow{\sim} \H^{\ast}(\bp(\ce)\times\bp(\ce), g_{1}^{\ast}\cf(1)),
\end{equation*}
and the latter cohomology group is in turn isomorphic to 
\begin{equation*}
\H^{\ast}(\bp(\ce), \cf(1))\otimes\H^{\ast}(\bp(\ce), \co_{\bp(\ce)}) \xrightarrow{\sim} \H^{\ast}(\S, \ce\otimes\cf)\otimes\H^{\ast}(\S, \co_{\S}). \qedhere
\end{equation*}
\end{proof}

\begin{remark}\label{euler}
\Cref{tautcoh} implies
\begin{equation*}
\chi(\Quot^{l}_{\S}(\ce), \cf^{[l]}) = \chi(\S, \ce\otimes\cf)\chi(\Quot^{l-1}_{\S}(\ce), \co_{\Quot^{l-1}_{\S}(\ce)}),
\end{equation*}
where $\chi(\Quot^{l-1}_{\S}(\ce), \co_{\Quot^{l-1}_{\S}(\ce)}) = \chi(\S^{[l-1]}, \co_{\S^{[l-1]}})$ is independent of $\ce$; in particular, the Euler characteristic $\chi(\Quot^{l}_{\S}(\ce), \cf^{[l]})$ is additive in $\ce$.
\end{remark}

\subsection{Further questions}

\Cref{tautcoh} is of course merely a special case of a much wider problem, 
which is to prove a Borel-Bott-Weil theorem for $\Quot^{l}_{\S}(\ce)$. 
Denote by $\Sigma^{\lambda}$ the Schur functor associated to a partition $\lambda$, 
so that in particular $\Sigma^{(n)}$ and $\Sigma^{(1^{n})}$ are the $n$-th symmetric and exterior powers, respectively. 
The problem is then to compute the cohomology groups
\begin{equation*}
\H^{\ast}(\Quot^{l}_{\S}(\ce), \otimes_{j=1}^{n}\Sigma^{\lambda_{j}}\cf_{j}^{[l]}),
\end{equation*}
where $\lambda_{1}, \ldots, \lambda_{n}$ are partitions and $\cf_{1}, \ldots, \cf_{n}$ are locally free sheaves on $\S$. 
Using \Cref{relations}, we show that these cohomology groups are compatible with properties (i) and (ii) of \Cref{I-III}.

\begin{proposition}\label{spectrals}
(i) For any locally free quotient $\ce\rightarrow\ce^{''}$ with kernel $\ce^{'}$, there is a fourth quadrant spectral sequence
\begin{equation*}
\E^{pq}_{1}=\H^{q}(\Quot^{l}_{\S}(\ce), \otimes_{j=1}^{n}\Sigma^{\lambda_{j}}\cf_{j}^{[l]} \otimes \Lambda^{-p}(\ce^{'\vee[l]})^{\vee})\Rightarrow\H^{p+q}(\Quot^{l}_{\S}(\ce^{''}), \otimes_{j=1}^{n}\Sigma^{\lambda_{j}}\cf_{j}^{[l]}).
\end{equation*}
(ii) For any invertible sheaf $\cl$ on $\S$, there is an isomorphism
\begin{equation*}
\H^{\ast}(\Quot^{l}_{\S}(\ce\otimes\cl), \otimes_{j=1}^{n}\Sigma^{\lambda_{j}}\cf_{j}^{[l]}) \xrightarrow{\sim} \H^{\ast}(\Quot^{l}_{\S}(\ce), \otimes_{j=1}^{n}\Sigma^{\lambda_{j}}(\cf_{j}\otimes\cl)^{[l]}).
\end{equation*}
\end{proposition}

\begin{proof}
(i) By \Cref{relations}, the Koszul complex
	\begin{equation*}
		\cdots\rightarrow \Lambda^{2}(\ce^{'\vee[l]})^{\vee}\rightarrow(\ce^{'\vee[l]})^{\vee}\rightarrow\co\rightarrow \co_{\Quot^{l}_{\S}(\ce^{''})}\rightarrow 0
	\end{equation*}
is exact. Tensoring with $\otimes_{j=1}^{n}\Sigma^{\lambda_{j}}\cf_{j}^{[l]}$, we obtain a resolution of $\otimes_{j=1}^{n}\Sigma^{\lambda_{j}}\cf_{j}^{[l]}\vert_{\Quot^{l}_{\S}(\ce^{''})}$; it remains to take the associated hypercohomology spectral sequence.
	
(ii) This is a direct consequence of \Cref{tautcomp} (ii).
\end{proof}
Part (i) implies the equality of Euler characteristics 
\begin{equation*}
\chi(\Quot^{l}_{\S}(\ce), \Sigma^{\lambda}\cf^{[l]}) = \chi(\Quot^{l}_{\S}(\ce^{''}), \Sigma^{\lambda}\cf^{[l]}) - \sum_{j\geqslant 1} (-1)^{j} \chi(\Quot^{l}_{\S}(\ce), \Sigma^{\lambda}\cf^{[l]} \otimes \Lambda^{j}(\ce^{'\vee[l]})^{\vee}).
\end{equation*}
for any partition $\lambda$. In the simplest case $\lambda=(1)$, combining this with \Cref{euler} yields
\begin{equation*}
-\sum_{j\geqslant 1} (-1)^{j} \chi(\Quot^{l}_{\S}(\ce), \cf^{[l]} \otimes \Lambda^{j}(\ce^{'\vee[l]})^{\vee}) = \chi(\Quot^{l}_{\S}(\ce^{'}), \cf^{[l]}).
\end{equation*}

\textbf{Acknowledgments.} The author is indebted to N. Arbesfeld, T. Schedler, B. Szendr\ho oi, 
and especially R. Thomas for several conversations and suggestions. 
He is also grateful to M. Artin, V. Ginzburg, and M. Hochster for helpful correspondence on commuting schemes. 
The results presented in this paper were found when the author was a research student at Imperial College, 
funded by Imperial College and the EPSRC [EP/L015234/1]; 
the exposition was amended when the author was supported by the grant  ERC-2017-AdG-786580-MACI. 
The project received funding from the European Research Council (ERC) under the European Union Horizon 2020 
research and innovation program (grant agreement No 786580).

\end{document}